\documentclass{mcom-l}
\usepackage{cleveref}
\usepackage{amsmath}
\usepackage{amsthm}
\usepackage{amssymb}
\usepackage{lipsum}
\usepackage{bm}
\usepackage{amsfonts}
\usepackage{graphicx}
\usepackage{epstopdf}
\usepackage{algorithmic}
\usepackage{xcolor}
\usepackage{float}

\newtheorem{theorem}{Theorem}[section]
\newtheorem{lemma}{Lemma}[section]
\newtheorem{proposition}[theorem]{Proposition}

\theoremstyle{definition}
\newtheorem{definition}[theorem]{Definition}
\newtheorem{example}[theorem]{Example}

\newtheorem{corollary}{Corollary}[section]

\newtheorem{remark}{Remark}[section]

\numberwithin{equation}{section}

\begin{document}

\title[QEO: Projection Method and Convergence Analysis]{Quasiperiodic Elliptic Operators: \\ Projection Method and Convergence Analysis}

\author{Kai Jiang}
\address{Hunan Key Laboratory for Computation and Simulation in Science and Engineering,
		Key Laboratory of Intelligent Computing and Information Processing of Ministry of Education, 
        School of Mathematics and Computational Science, Xiangtan University, Xiangtan, Hunan, China, 411105.}
\curraddr{}
\email{kaijiang@xtu.edu.cn}
\thanks{\textbf{Funding:} This work is supported by the National Key R\&D Program of China (2023YFA1008800). Kai Jiang is supported in part by the National Natural Science Foundation of China (12288101), the Science and Technology Innovation Program of Hunan Province (2024RC1052), the Innovative Research Group Project of Natural Science Foundation of Hunan Province of China (2024JJ1008). Qilong Zhai is supported in part by the National Natural Science Foundation of China (12271208).}

\author{Jiaqi Tang}
\address{School of Mathematics, Jilin University, Changchun, Jilin, China, 130012}
\curraddr{}
\email{tangjq23@mails.jlu.edu.cn}
\thanks{}

\author{Qilong Zhai}
\address{School of Mathematics, Jilin University, Changchun, Jilin, China, 130012}
\curraddr{}
\email{zhaiql@jlu.edu.cn}

\author{Qi Zhou}
\address{School of Mathematics and Computational Science, Xiangtan University, Xiangtan, Hunan, China, 411105.}
\curraddr{Institute for Math and AI, Wuhan University, Wuhan, Hubei, China, 430072.}
\email{qizhou@smail.xtu.edu.cn}
\thanks{}


\date{}

\dedicatory{}

\begin{abstract}
    Quasiperiodic elliptic operators (QEOs) serve as fundamental models in both mathematics and physics, as exemplified by their role in the numerical modeling of one-dimensional photonic quasicrystals. 
	However, distinct from periodic elliptic operators, approximating eigenpairs for QEOs poses significant challenges, particularly in capturing the full spectral structure (notably the continuous spectrum) and deriving convergence guarantees in the absence of compactness. 
	In this paper, we develop a high-accuracy numerical method to compute eigenpairs of QEOs based on the projection method, which embeds quasiperiodic operators into a higher-dimensional periodic torus. To address the non-compactness issue, we construct a directional-derivative Hilbert space along irrational manifolds of a high-dimensional torus and characterize operators equivalent to QEOs within this space. By integrating spectral theory for non-compact operators into the Babu\v{s}ka--Osborn eigenproblem framework, we establish rigorous convergence analysis and prove that our method achieves spectral accuracy. Numerical experiments validate the accuracy and efficiency of the proposed method, including a one-dimensional photonic quasicrystal and two- and three-dimensional QEOs.
\end{abstract}

\maketitle
\section{Introduction}
\subsection{Background}
In this paper, we consider the $d$-dimensional quasiperiodic elliptic operator (QEO)
\begin{equation*}
L := -\nabla \cdot \big( \alpha(\bm{x}) \nabla \big),\qquad \bm{x} \in \mathbb{R}^d,
\end{equation*}
where the coefficient $\alpha(\bm{x})$ is a quasiperiodic function (see Definition \ref{def:QP}). 
This operator is a central object of study, with substantial significance in both mathematical analysis and physical modeling. 

Mathematically, QEOs present a considerable leap beyond periodic elliptic operators (PEOs), as quasiperiodicity leads to distinct   {compactness properties and spectral structure}. 
PEOs possess compact resolvents on periodic function spaces, which ensures that their spectra consist solely of discrete eigenvalues \cite{bruning1992spectrum,kuchment2016overview}.
By contrast,  QEOs lack this compactness property because compact embedding theorems fail in Besicovitch quasiperiodic function spaces \cite{besicovitch1954almost}.
Consequently, QEO spectra may contain continuous components, in sharp contrast to the purely discrete spectra of their periodic counterparts.
This theoretical gap leaves QEOs without an established spectral-analysis framework, making  {the rigorous numerical computation and analysis of their eigenpairs} a  {highly nontrivial} problem.

 {
	Moreover, QEOs also serve as fundamental mathematical and physical models. In the one-dimensional setting, they underpin the numerical realization of photonic quasicrystals \cite{chan1998photonic,vardeny2013optics}.
}  
{Distinct from periodic structures, the inherent quasiperiodicity endows them with remarkable physical behaviors,} including more stable superconducting pairing, defect resistance, and dissipationless transmission \cite{levine1984,man2005experimental,vyunishev2017quasiperiodic}. Consequently, these attributes have in recent years driven the incorporation of quasiperiodic structures into advanced photonic devices, such as broadband optical filters, low-loss optical fibers, and high-efficiency lasers.

\subsection{Numerical progress and challenges}

The eigenproblem of QEOs poses substantial challenges for both numerical computation and theoretical analysis, owing to the non-compactness of  {these operators} and the non-decaying, space-filling nature of their  {eigenfunctions}.  
A traditional approach, the periodic approximation method (PAM), tackles such problems by approximating quasiperiodic systems with periodic ones~\cite{zhang2008efficient}. 
The accuracy of PAM is dominated by Diophantine errors arising from approximating irrational numbers by rational ones, which precludes uniform error decay~\cite{jiang2023}.
When applied to QEOs,  {PAM suffers from a fundamental flaw by approximating QEOs as PEOs. Since the spectra of PEOs consist solely of discrete eigenvalues, PAM can only compute discrete spectra and fails to capture any continuous-spectrum information}~\cite{che2021polarization,zhang2022unfolded}. 
 {Even as the approximation period is enlarged (i.e., the computational domain is expanded) and the discretization is refined, PAM merely refines existing discrete eigenvalues rather than inserting new spectral components into the gaps.} 
 {Consequently, PAM is inherently incapable of reflecting the spectral features of QEOs.}

Recently, the projection method (PM) has emerged as a highly accurate and efficient approach for quasiperiodic systems \cite{jiang2014}, treating each quasiperiodic system as an irrational manifold of a higher-dimensional periodic system.
Extensive studies have shown the accuracy and efficiency of PM for a broad range of quasiperiodic problems, including quasicrystals~\cite{jiang2015stability,si2025designing}, incommensurate quantum systems~\cite{jiang2024irrational}, quasiperiodic multiscale problems~\cite{jiang2024}, and grain boundaries~\cite{jiang2022tilt}. 
Theoretical analyses of PM for solving differential equations have also been progressively developed~\cite{jiang2024numerical,jiang2024}. 
 {When applied to QEOs, PM is expected to overcome the limitation of PAM, precisely because its high-dimensional embedding captures the global structure of quasiperiodicity. Rather than being restricted to preexisting discrete eigenvalues, PM can, under mesh refinement, insert new eigenvalues into spectral gaps and reveal the underlying continuous bands. As numerical accuracy increases, these gaps are progressively filled, suggesting a dense approximation of the overall spectral structure~\cite{rodriguez2008computation, quan2026spectral}.}
Despite these advantages, a mature theoretical framework for the convergence of PM on QEOs is still lacking. Establishing such rigorous convergence analysis requires addressing the fundamental difficulty posed by non-compactness. Unlike periodic elliptic operators, the absence of compactness in QEOs invalidates the operator compact approximation theory that underpins the Galerkin method~\cite{vainikko1969compact,krasnosel2012approximate,zhai2019}, making the theoretical analysis of PM for QEOs substantially more challenging. 

\subsection{Contribution}
In this work, we focus on the numerical computation of eigenpairs of QEOs and provide the corresponding convergence analysis. 
The main contributions are summarized as follows. 
 {
\begin{itemize}
    \setlength{\itemsep}{0.45em}
    \item We propose a PM-based algorithm for computing QEO eigenpairs by embedding quasiperiodic operators into a higher-dimensional periodic torus and solving the equivalent problem with Fourier discretization. Moreover, to address the computational challenges caused by the dimensionality lifting, the fast Fourier transform is incorporated to improve computational efficiency. 
    \item We provide the theoretical analysis for the proposed algorithm by constructing a directional-derivative Hilbert space along irrational manifolds of the high-dimensional torus. Building on this formulation, we integrate spectral theory for non-compact operators~\cite{descloux19782,descloux1978} into the Babu\v{s}ka--Osborn eigenproblem framework~\cite{babuvska1991}, thereby resolving the non-compactness obstacle and establishing spectral-accuracy convergence for both eigenvalues and eigenfunctions.
    \item  We adopt an efficient diagonal preconditioner to mitigate ill-conditioning in the discretized systems and systematically validate the proposed method through experiments, including one-dimensional photonic quasicrystals and two- and three-dimensional cases. Numerical results show that our method has spectral convergence and higher accuracy compared to PAM.
\end{itemize}
}

\subsection{Organization}
The rest of this paper is organized as follows. In \Cref{sec:pre}, we present preliminaries on quasiperiodic function spaces. In \Cref{sec:eig}, we propose a method for QEOs and provide a rigorous convergence and error analysis. In \Cref{sec:exp}, we show accuracy and efficiency through numerical examples of a one-dimensional photonic quasicrystal and two- and three-dimensional QEOs. Finally, \Cref{sec:con} concludes  {this work} and outlines the future work.

\subsection{Notations}
We collect the main notations used throughout the paper. For a vector $\bm{\alpha}=(\alpha_i)^n_{i=1}\in \mathbb{R}^n$, 
$|\bm{\alpha}|:=\alpha_1+\cdots+\alpha_n$. For two numbers $a$ and $b$, we write $a \lesssim b$ ($a \gtrsim b$) if there exists a constant $C>0$ such that $a \le C b$ ($a \ge C b$). For matrices $\bm{A}=(a_{ij}), \bm{B}=(b_{ij}) \in \mathbb{C}^{n \times n}$, their Hadamard product is denoted by $\bm{A} \circ \bm{B}$, with $(\bm{A} \circ \bm{B})_{ij} = a_{ij}b_{ij}$. The Frobenius norm of $\bm{A}=(a_{ij}) \in \mathbb{C}^{n \times n}$ is $\|\bm{A}\|_F := \bigl( \sum_{i=1}^{n}\sum_{j=1}^{n} |a_{ij}|^2 \bigr)^{1/2}$. $\{\bm{e}_i\}_{i=1}^{n}$ is the canonical basis of $\mathbb{R}^n$.

\section{Preliminaries}
\label{sec:pre}

	In this section, we provide a brief review of quasiperiodic functions and corresponding function spaces. 
   	
\subsection{Quasiperiodic functions}
In this subsection, we present the definition of quasiperiodic functions and establish their relationship with the corresponding parent periodic functions.
	
	\begin{definition}
	A matrix $\bm{P}$ is called the projection matrix if it belongs to the set $\mathbb{P}^{d\times n} :=\{\bm{P}=(\bm{p}_1,\cdots,\bm{p}_n)\in\mathbb{R}^{d\times n}:
    \bm{p}_1,\cdots,\bm{p}_n~\mbox{are}~\mathbb{Q}\mbox{-linearly}~\mbox{independent},\\ {\rm rank}_{\mathbb{R}}(\bm{P})=d\}.$
	\end{definition}
	\begin{definition}\label{def:QP}
		A $d$-dimensional function $f(\bm{x})$ is quasiperiodic if there exists a continuous $n$-dimensional periodic function $F(\bm{y})$ and a projection matrix $\bm{P} \in \mathbb{P}^{d \times n}$ such that $f(\bm{x}) = F(\bm{P}^{T}\bm{x})$ for all $\bm{x} \in \mathbb{R}^{d}$. Here, $F(\bm{y})$ is called the parent function of $f(\bm{x})$. 
	\end{definition}
	\noindent
	In particular, when $n=d$, $f(\bm{x})$ is periodic; when $n\to\infty$, $f(\bm{x})$ is almost periodic~\cite{bohr2018almost,besicovitch1954almost}. The set of all quasiperiodic functions is denoted by ${\rm QP}(\mathbb{R}^{d})$. 

    For a quasiperiodic function $f\in {\rm QP}(\mathbb{R}^d)$, its mean value $\mathcal{M}\{f\}$ is defined as 
	$$\mathcal{M} \{ f \} := \lim\limits_{T \to \infty} \frac{1}{(2T)^{d}}\int_{\bm{s}+[-T,T]^d}f(\bm{x}){d}\bm{x},\qquad \forall\,\bm{s}\in\mathbb{R}^d.$$
	Since quasiperiodic functions are defined on the entire space $\mathbb{R}^d$, the Fourier-Bohr series of $f(\bm{x})$ is defined by
	$$f(\bm{x}) = \sum_{\bm{\lambda_{k}} \in \Lambda}\hat{f}_{\bm{\lambda_{k}}}e^{i\bm{\lambda_{k}}^{T}\bm{x}},\qquad \hat{f}_{\bm{\lambda_{k}}} = \mathcal{M}\big\{f(\bm{x})e^{-i \bm{\lambda}_{\bm k}^{T} \bm{x}}\big\},$$
	where $\Lambda = \{\bm{\lambda}:\bm{\lambda} = \bm{Pk},~\bm{k} \in \mathbb{Z}^n\}$ are Fourier exponents. Fourier bases satisfy the following orthogonality 
	\begin{equation}\label{eqn-05221}
\mathcal{M}\{e^{i\bm{\lambda}_{1}^{T}\bm{x}}e^{-i\bm{\lambda}_{2}^{T}\bm{x}}\} = 
	\begin{cases}
		1,& \bm{\lambda}_{1} = \bm{\lambda}_{2},\\
		0,& {\rm otherwise}.
	\end{cases}
	\end{equation}
	
	Let $\mathbb{T}^{n} = (\mathbb{R}/ 2\pi \mathbb{Z})^{n}$ be the $n$-dimensional torus. The Fourier series of the $n$-dimensional periodic function $F(\bm{y})$ defined on $\mathbb{T}^{n}$ is
	$$F(\bm{y}) = \sum_{\bm{k} \in \mathbb{Z}^{n}}\hat{F}_{\bm{k}}e^{i\bm{k}^{T}\bm{y}},\qquad\hat{F}_{\bm{k}} = \frac{1}{\left| \mathbb{T}^{n} \right|}\int_{\mathbb{T}^{n}}F(\bm{y})e^{-i\bm{k}^{T}\bm{y}}{d}\bm{y},~ \bm{k} \in \mathbb{Z}^{n}.$$ 
    There is a close connection between quasiperiodic functions and their parent functions, which can be used to simplify the analysis of quasiperiodic systems~\cite{fan2025representation}. We present the following theorem, which essentially relies on the ergodicity induced by irrationality. 
	\begin{lemma}[\cite{jiang2024}, Theorem 2.3]\label{thm-1}
		For a given continuous quasiperiodic function
		$f(\bm{x}) = F(\bm{P}^{T}\bm{x}),~\bm{x} \in \mathbb{R}^{d},$
		where $F(\bm{y})\in\mathcal{C}(\mathbb{T}^n)$ is the  parent function and $\bm{P}$ is the projection matrix, we have
		$\hat{f}_{\bm{\lambda_{k}}} = \hat{F}_{\bm{k}},~\forall\,\bm{k}\in\mathbb{Z}^n.$
	\end{lemma}

    \subsection{Besicovitch quasiperiodic function spaces}

    The Hilbert space defined on $\mathbb{T}^{n}$ is
	$$\mathcal{L}^{2}(\mathbb{T}^{n}) = \big\{ F(\bm{y}): (F,F) <+\infty \big\},~ (F_{1},F_{2}) = \frac{1}{\left| \mathbb{T}^{n} \right|}\int_{\mathbb{T}^{n}}F_{1}\bar{F}_{2}{\rm d}\bm{y}.$$
	For any $s \in \mathbb{N}^{+}$, the $s$-derivative Sobolev space is
	$$\mathcal{H}^{s}(\mathbb{T}^{n}) = \big\{ F \in \mathcal{L}^{2}(\mathbb{T}^{n}):\| F\|_{\mathcal{H}^{s}(\mathbb{T}^{n})} < \infty \big\}, $$
	equipped with the norm
	$$\|F\|_{\mathcal{H}^{s}(\mathbb{T}^{n})} = \left( \sum_{\bm{k} \in \mathbb{Z}^{n}}\left(1+\left|\bm{k}\right|^{2}\right)^{s}\left|\hat{F}_{\bm{k}}\right|^{2}\right)^{1/2}.$$
    
    We now introduce the Besicovitch quasiperiodic function spaces. The inner product space on ${\rm QP}(\mathbb{R}^{d})$ is defined as
	$$\mathcal{L}^{2}_{\rm QP}(\mathbb{R}^{d})=\left\{f \in {\rm QP}(\mathbb{R}^{d}):\mathcal{M}\{\left|f\right|^{2}\} < \infty \right\},$$
	which is equipped with the inner product and the norm
	$$\left(f_{1},f_{2}\right)_{\mathcal{L}^{2}_{\rm QP}(\mathbb{R}^{d})}=\mathcal{M}\{f_{1}(x) \overline{f_{2}(x)}\}, \quad \|f\|_{\mathcal{L}^{2}_{\rm QP}(\mathbb{R}^{d})}=\left( f,f\right) ^{1/2}_{\mathcal{L}^{2}_{\rm QP}(\mathbb{R}^{d})}.$$
	Similarly, we define
	$$\mathcal{L}^{\infty}_{\rm QP}(\mathbb{R}^{d}) = \left\{f \in {\rm QP}(\mathbb{R}^{d}): \sup_{\bm{x}\in \mathbb{R}^{d}}\left|f(\bm{x})\right| < \infty \right\}.$$
	For any integer $s \geq 0$, the $s$-derivative Sobolev space $\mathcal{H}^{s}_{\rm QP}(\mathbb{R}^{d})$ is defined through the following norm
	$$\|f\|^2_{\mathcal{H}^{s}_{\rm QP}(\mathbb{R}^{d})}=\sum_{\left|\bm{p}\right| \leq s}\left(\partial^{\bm{p}}_{\bm{x}}f_{1},\partial^{\bm p}_{\bm{x}}f_{2}\right)_{\mathcal{L}^{2}_{\rm QP}(\mathbb{R}^{d})},$$
    and semi-norm 
	$$\left|f\right|^{2}_{\mathcal{H}^{s}_{\rm QP}(\mathbb{R}^{d})} = \sum_{\left|{\bm p}\right|=s}\left(\partial^{\bm p}_{\bm{x}}f,\partial^{\bm p}_{\bm{x}}f\right)_{\mathcal{L}^{2}_{\rm QP}(\mathbb{R}^{d})}.$$
	Moreover, Parseval's identity in the $\mathcal{L}^{2}$-norm sense is given by
	$$\|f\|_{\mathcal{L}^{2}_{\rm QP}(\mathbb{R}^{d})} = \left(\sum_{\bm{\lambda}_{\bm k} \in \Lambda}\left|\hat{f}_{\bm{\lambda}_{\bm k}}\right|^{2}\right)^{1/2},$$
	along with the corresponding $\mathcal{H}^{s}_{\rm QP}(\mathbb{R}^{d})$ norm and semi-norm
	$$\|f\|_{\mathcal{H}^{s}_{\rm QP}(\mathbb{R}^{d})} = \left(\sum_{\bm{\lambda}_{\bm k} \in \Lambda}\left(1+\left|\bm{\lambda}_{\bm k}\right|^{2}\right)^{s}\left|\hat{f}_{\bm{\lambda}_{\bm k}}\right|^{2}\right)^{1/2},$$
    $$|f|_{\mathcal{H}^{s}_{\rm QP}(\mathbb{R}^{d})} = \left(\sum_{\bm{\lambda}_{\bm k} \in \Lambda}\left|\bm{\lambda}_{\bm k}\right|^{2s}\left|\hat{f}_{\bm{\lambda}_{\bm k}}\right|^{2}\right)^{1/2}.$$

    \begin{remark}\label{rem-1}
    The classical Besicovitch almost periodic function space forms a Hilbert space, whose frequency spectrum is generated by countably many $\mathbb{Q}$-linearly independent vectors~\cite{besicovitch1954almost}. 
    In numerical calculations, we can only consider quasiperiodic functions with a finitely generated Fourier--Bohr spectrum (\textit{i.e.}, the projection matrix has finitely many columns), which leads to incompleteness.
    Consequently, the Besicovitch quasiperiodic function space used here is not a Hilbert space. 
    \end{remark}

\section{Numerical method and theoretical framework for QEOs}
\label{sec:eig}
In this section, we present the numerical scheme based on the PM for QEOs. Moreover, we establish a complete theoretical framework for the convergence analysis and error estimates of the eigenvalues and corresponding eigenfunctions.

	\subsection{The eigenproblem of QEOs}
	To ensure the well-posedness of the solution, we consider the following eigenproblem of QEOs. 

Find eigenpair $(\gamma,u)$, $\|u\|_{\mathcal{L}^{2}_{\rm QP}(\mathbb{R}^{d})}=1$, such that
	\begin{equation}\label{eqn-4}
			(L+I)u(\bm{x})=-\nabla \cdot (\alpha(\bm{x})\nabla u(\bm{x}))+u(\bm{x})=\gamma u(\bm{x}), \qquad\bm{x} \in \mathbb{R}^{d}.
	\end{equation} 
    Denote $L_1:=-\nabla \cdot \big( \alpha(\bm{x}) \nabla \big)+I$. It is evident that $L_1$ and $L$ commute and thus share the same eigenfunctions. The eigenvalues of $L_1$ are those of the original QEO shifted by one.
	The coefficient $\alpha(\bm{x})\in \mathcal{L}^{\infty}_{\rm QP}(\mathbb{R}^{d})$ is uniformly elliptic and bounded, \textit{i.e.}, for all $\bm{x},\bm{\xi} \in \mathbb{R}^{d}$, there are constants $C_{0},C_{1}>0$ such that
$$
C_{0}\left| \bm{\xi} \right|^{2} \leq \bm{\xi}^{T}\alpha(\bm{x})\bm{\xi} \leq C_{1}\left| \bm{\xi} \right|^{2}.
$$

Further, the variational form of \eqref{eqn-4} is to find $\gamma \in \mathbb{R}$ and $u \in \mathcal{H}^{1}_{\rm QP}(\mathbb{R}^{d})$ such that $\|u\|_{\mathcal{H}^{1}_{\rm QP}(\mathbb{R}^{d})}=1$ and 
	\begin{equation}\label{eq:QP_varia}
		(\alpha \nabla u,\nabla v)+(u,v)=\gamma(u,v),\qquad \forall\,v \in \mathcal{H}^{1}_{\rm QP}(\mathbb{R}^{d}).
	\end{equation}
As noted in \Cref{rem-1}, $\mathcal{H}^{1}_{\rm QP}(\mathbb{R}^{d})$ is not a Hilbert space. We introduce a new inner product space on $\mathbb{T}^{n}$ tailored to the spectral properties of quasiperiodic functions
$$\mathcal{H}^{1}_{\bm{P}}(\mathbb{T}^{n})=\left\{U \in \mathcal{L}^{2}(\mathbb{T}^{n}),~ \bm{P}\nabla U \in [\mathcal{L}^{2}(\mathbb{T}^{n})]^{d}\right\},$$
which is equipped with the inner product and the norm
\begin{equation}\label{eqn-07302}
(U,V)_{\mathcal{H}_{\bm{P}}^{1}(\mathbb{T}^{n})} = (U,V)+(\bm{P}\nabla U,\bm{P}\nabla V),~ \|U\|_{\mathcal{H}_{\bm{P}}^{1}(\mathbb{T}^{n})} = (U,U)_{\mathcal{H}_{\bm{P}}^{1}(\mathbb{T}^{n})}^{1/2}.
\end{equation}
	\begin{proposition}\label{thm-com}
		$\mathcal{H}^{1}_{\bm{P}}(\mathbb{T}^{n})$ is a Hilbert space.
	\end{proposition}
     \begin{proof}
        Let $\{U_{m}\}$ be a Cauchy sequence in $\mathcal{H}_{\bm{P}}^{1}(\mathbb{T}^{n})$. By definition of the $\mathcal{H}_{\bm{P}}^{1}(\mathbb{T}^{n})$ norm, both $\{U_{m}\}$ and $\{\bm{P} \nabla U_{m}\}$ are Cauchy sequences. Since $\mathcal{L}^{2}(\mathbb{T}^{n})$ is a Hilbert space, there exist limits
$$U_{m} \to U~ in~ \mathcal{L}^{2}(\mathbb{T}^{n}),$$
$$\bm{P} \nabla U_{m} \to W~ in~ [\mathcal{L}^{2}(\mathbb{T}^{n})]^{d}.$$
Fix a test function $\Phi\in \mathcal{H}^{1}_{\bm{P}}(\mathbb{T}^{n})$, and by the definition of weak derivatives, 
$$(W,\Phi) = \lim\limits_{m \to \infty}(\bm{P}\nabla U_{m},\Phi) = \lim\limits_{m \to \infty}-(U_{m},\bm{P} \nabla \Phi)=-(U,\bm{P}\nabla \Phi)=(\bm{P}\nabla U,\Phi).$$
Thus, $W$ is the weak derivative $\bm{P}\nabla U$ and $U \in \mathcal{H}^{1}_{\bm{P}}(\mathbb{T}^{n})$. Therefore,
$$U_{m} \to U~ in~ \mathcal{H}_{\bm{P}}^{1}(\mathbb{T}^{n}).$$
\end{proof}
 
Then we can define the $\mathcal{H}^{\alpha}_{\bm{P}}(\mathbb{T}^{n})$ norm
	$$\|F\|_{\mathcal{H}^{\alpha}_{\bm{P}}(\mathbb{T}^{n})} = \sum_{\bm{k} \in \mathbb{Z}^{n}}\left(1+\left|\bm{P}\bm{k}\right|^{2}\right)^{\alpha}|\hat{F}_{\bm{k}}|^{2}.$$
	Correspondingly, we define the dual spaces $\mathcal{H}_{\bm{P}}^{-\alpha}(\mathbb{T}^{n}) = (\mathcal{H}_{\bm{P}}^{\alpha}(\mathbb{T}^{n}))^{'}$ with the norm
	$$\|F\|_{-\alpha} = \sup_{G \in \mathcal{H}_{\bm{P}}^{\alpha}(\mathbb{T}^{n})}\frac{(G,F)}{\|F\|_{\alpha}}.$$
To simplify notation, we use $\| \cdot \|_{\alpha}$ to denote the norm $\| \cdot \|_{\mathcal{H}^{\alpha}_{\bm{P}}(\mathbb{T}^{n})}$. 

By \Cref{thm-1} and the orthogonality of Fourier bases, we can derive the modal equation of (\ref{eq:QP_varia}) as follows
$$\sum_{\bm{k}_{\mathcal{V}} \in \mathbb{Z}^{n}}\sum_{\bm{k}_{\mathcal{U}} \in \mathbb{Z}^{n}}\hat{\alpha}_{\bm{Pk}_{\mathcal{A}}}(\bm{Pk}_{\mathcal{V}})^{T}(\bm{Pk}_{\mathcal{U}})\hat{u}_{\bm{Pk}_{\mathcal{U}}}\hat{v}_{\bm{Pk}_{\mathcal{V}}}+\sum_{\bm{k}_{\mathcal{V}} \in \mathbb{Z}^{n}}\hat{u}_{\bm{Pk}_{\mathcal{V}}}\hat{v}_{\bm{Pk}_{\mathcal{V}}} = \sum_{\bm{k}_{\mathcal{V}} \in \mathbb{Z}^{n}} \gamma \hat{u}_{\bm{Pk}_{\mathcal{V}}}\hat{v}_{\bm{Pk}_{\mathcal{V}}},$$ $$\bm{k}_{\mathcal{A}} = \bm{k}_{\mathcal{V}}-\bm{k}_{\mathcal{U}}.$$
The above expression coincides with the formulation of the high-dimensional periodic system below.

Find $\gamma \in \mathbb{R},~ \mathcal{U} \in \mathcal{H}^{1}_{\bm{P}}(\mathbb{T}^{n})$, such that $\|\mathcal{U}\|_{1}=1,$ and 
\begin{equation}\label{eqn-5}
	a(\mathcal{U},\mathcal{V})=\gamma(\mathcal{U},\mathcal{V}),\qquad \forall\,\mathcal{V} \in \mathcal{H}^{1}_{\bm{P}}(\mathbb{T}^{n}),
\end{equation}
where $\mathcal{A},\mathcal{U},\mathcal{V}$ are the parent functions of $\alpha,u,v$, respectively, and
$$a(\mathcal{U},\mathcal{V})=(\mathcal{A} \bm{P} \nabla \mathcal{U},\bm{P}\nabla \mathcal{V})+(\mathcal{U},\mathcal{V}).$$

The source problem associated with \eqref{eqn-5} is as follows: Given $\mathcal{F}\in\mathcal{L}^2(\mathbb{T}^n)$, find $\mathcal{W}\in\mathcal{H}^{1}_{\bm{P}}(\mathbb{T}^{n})$ satisfying
\begin{equation}\label{eq:source}
	a(\mathcal{W},\mathcal{V})=(\mathcal{F},\mathcal{V}),\qquad \forall\,\mathcal{V}\in\mathcal{H}^{1}_{\bm{P}}(\mathbb{T}^{n}).
\end{equation}

\subsection{Discretize QEOs by PM}
    Now, we discretize the equivalence problem \eqref{eqn-5} by PM. 
	For any $N \in \mathbb{N}^{+}$ and a given projection matrix $\bm{P} \in \mathbb{P}^{d \times n}$, denote
	$$K_{N}^{n} = \left\{ \bm{k} = (k_{j})_{j=1}^{n} \in \mathbb{Z}^{n}:-N/2 \leq k_{j} < N/2,~j = 1,\cdots,n\right\}.$$
	For periodic systems, the spectral collocation method provides an efficient way to determine Fourier coefficients. To simplify our discussion, we will consider a fundamental domain $[0,2 \pi)^{n}$ and assume that the discretization is uniform along each dimension. Namely, we discretize the torus $\mathbb{T}^{n}$ as
	$$\mathbb{T}_{N}^{n} = \left\{\bm{y_{j}}=(2\pi j_{1}/{N},\cdots,2\pi j_{N}/{N}) \in \mathbb{T}^{n}:0\leq j_1,\cdots,j_n\leq N\right\}.$$
	We denote the grid function space as
	$$\mathcal{G}_{N} = \{F: \mathbb{T}_{N}^{n} \to \mathbb{C}:~ F ~ is~ \mathbb{T}_{N}^{n}-{\rm periodic}\}.$$
	For any $\mathbb{T}_{N}^{n}$-periodic grid functions $G$, $F$, the $\ell_{2}$-inner product is defined as
\begin{equation}\label{eqn-07291}
(G,F)_{N} = \frac{1}{N^{n}}\sum_{\bm{y_{j}}\in \mathbb{T}_{N}^{n}}G(\bm{y_{j}})\bar{F}(\bm{y_{j}}).
\end{equation}
	For any $\bm{k},~\bm{l}\in\mathbb{Z}^n$, a simple calculation yields
\begin{equation}\label{eqn-05222}
	(e^{i\bm{k}^{T}\bm{y}},e^{i\bm{l}^{T} \bm{y}})_{N}=
    \begin{cases}
        1,& \bm{k} = \bm{l}+N\bm{m},~ \bm{m} \in \mathbb{Z}^{n},\\
        0,& {\rm otherwise}.
    \end{cases}
\end{equation}

We then give the discrete form of \eqref{eqn-5} using PM. Find $\gamma_{N} \in \mathbb{R}$ and $\mathcal{U}_{N} \in V^{N}={\rm span}\{e^{i\bm{k}^T\bm{y}}:\bm{k} \in K^{n}_{N},\bm{y} \in \mathbb{T}^{n}\}$ such that 
	\begin{equation}\label{eqn-6}
		a_{p}(\mathcal{U}_{N},\mathcal{V}_{N})=\gamma_{N}(\mathcal{U}_{N},\mathcal{V}_{N})_{N},\qquad\forall\,\mathcal{V}_{N} \in V^{N},
	\end{equation}
	where
	$$a_{p}(\mathcal{U}_{N},\mathcal{V}_{N})=(\mathcal{A}\bm{P} \nabla \mathcal{U}_{N},\bm{P}\nabla \mathcal{V}_{N})_{N}+(\mathcal{U}_{N},\mathcal{V}_{N})_{N}.$$ 
	Note that $(\cdot,\cdot)_{N}$ and $(\cdot,\cdot)$ coincide on $\mathcal{V}_{N}$, \textit{i.e.},
	\begin{equation}\label{eqn-7}
		(\mathcal{U}_{N},\mathcal{V}_{N})_{N} = (\mathcal{U}_{N},\mathcal{V}_{N}),\qquad \forall\,\mathcal{U}_{N},\mathcal{V}_{N}\in V^{N},
	\end{equation}
	which is based on the numerical integration formula
	$$\frac{1}{\left| \mathbb{T}^{n} \right|}\int_{\mathbb{T}^{n}}F{\rm d}\bm{y} \approx \frac{1}{N^{n}}\sum_{\bm{y_{j}}\in \mathbb{T}_{N}^{n}}{F}(\bm{y_{j}}).$$

The discrete form of the source problem \eqref{eq:source} is given by: Find $\mathcal{W}_N\in V^N$, such that
	\begin{equation}\label{eq:source_discrete}
		a(\mathcal{W}_N,\mathcal{V})=(\mathcal{F},\mathcal{V}),\qquad \forall\,\mathcal{V}\in V^N.
	\end{equation}

    \begin{lemma}\label{lem-05231}
		For any $\mathcal{U}_{N} \in V^{N}$,
		$a_{p}(\mathcal{U}_{N},\mathcal{U}_{N}) \gtrsim \|\mathcal{U}_{N}\|^{2}_{1}.$
	\end{lemma}
	\begin{proof}
		Based on (\ref{eqn-7}), we have
		 \begin{equation*}
			\begin{aligned}
				a_{p}(\mathcal{U}_{N},\mathcal{U}_{N})& \gtrsim (\bm{P}\nabla \mathcal{U}_{N},\bm{P}\nabla \mathcal{U}_{N})_{N}+(\mathcal{U}_{N},\mathcal{U}_{N})_{N}\\
				&=(\bm{P}\nabla \mathcal{U}_{N},\bm{P}\nabla \mathcal{U}_{N})+(\mathcal{U}_{N},\mathcal{U}_{N}) \\
				&= \|\mathcal{U}_{N}\|^{2}_{1},
			\end{aligned}
		 \end{equation*}
    which implies that $a_{p}(\cdot,\cdot)$ is coercive.
	\end{proof}
    
	The discrete Fourier series of high-dimensional periodic functions $\mathcal{U}(\bm{y})$ and $\mathcal{A}(\bm{y})$ are given by
    $$\mathcal{U}(\bm{y})=\sum_{\bm{k}_\mathcal{U}\in K^n_N}\widetilde{\mathcal{U}}_{\bm{k}_\mathcal{U}}e^{i\bm{k}_\mathcal{U}^{T}\bm{y_{j}}},~ \tilde{\mathcal{U}}_{\bm{k}_{\mathcal{U}}}= \frac{1}{N^{n}}\sum_{\bm{y_{j}}}\mathcal{U}(\bm{y_{j}})e^{-i \bm{k}^{T}_{\mathcal{U}}\bm{y_{j}}},~\bm{y_{j}}\in \mathbb{T}_{N}^{n},$$
    $$\mathcal{A}(\bm{y}_{\bm{j}})=\sum_{\bm{k}_\mathcal{A}\in K^n_N}\widetilde{\mathcal{A}}_{\bm{k}_\mathcal{A}}e^{i\bm{k}_\mathcal{A}^{T}\bm{y_{j}}},~\widetilde{\mathcal{A}}_{\bm{k}_\mathcal{A}} = \frac{1}{N^{n}}\sum_{\bm{y_{j}} \in \mathbb{T}_{N}^{n}}\mathcal{A}(\bm{y_{j}})e^{-i \bm{k}_\mathcal{A}^{T}\bm{y_{j}}},~\bm{y_{j}}\in \mathbb{T}_{N}^{n}.$$
    Then we can derive the discrete scheme of \eqref{eqn-6} as
\begin{equation}\label{eq:PM_discrete}
    \sum_{\bm{k}_{\mathcal{U}} \in K_{N}^{n}}\tilde{\mathcal{A}}_{\bm{k}_{\mathcal{A}}}(\bm{Pk}_{\mathcal{V}})^{T}(\bm{Pk}_{\mathcal{U}})\tilde{\mathcal{U}}_{\bm{k}_{\mathcal{U}}}+ \tilde{\mathcal{U}}_{\bm{k}_{\mathcal{V}}}= \gamma \tilde{\mathcal{U}}_{\bm{k}_{\mathcal{V}}},\quad \bm{k}_{\mathcal{A}} = (\bm{k}_{\mathcal{V}}-\bm{k}_{\mathcal{U}})({\rm mod}~ N).
\end{equation}
  
	Furthermore, we can generate the matrix eigenvalue problem using the tensor-vector-index conversion \cite{jiang2024}. Define matrices $\bm{A}=(A_{ij})$ and $\bm{W}=(W_{ij})$ as follows
	$$A_{ij} = \tilde{\mathcal{A}}_{\bm{k}_{\mathcal{A}}},~\quad W_{ij} = (\bm{Pk}_{\mathcal{V}})^{T}(\bm{Pk}_{\mathcal{U}}).$$
	The column vectors $\bm{U}=(U_j)$ are defined by
	$U_{j} = \tilde{\mathcal{U}}_{\bm{k}_{\mathcal{U}}}.$
	Consequently, we obtain the following linear system for \eqref{eq:PM_discrete}
    \begin{equation}\label{eq:lin_sys}
        \bm{QU} = \gamma \bm{U},~ \bm{Q} = \bm{A} \circ \bm{W}.
    \end{equation}
    
    The truncated and interpolated projection operators of PM are defined by
		$$\mathcal{P}_{N}\mathcal{U}(\bm{x})=\sum_{\left|\bm{k}\right| \leq N}\hat{\mathcal{U}}_{\bm{k}}e^{i\bm{k}^T\bm{x}},\quad \mathcal{I}_{N}\mathcal{U}(\bm{x})=\sum_{\left|\bm{k}\right| \leq N}\tilde{\mathcal{U}}_{\bm{k}}e^{i\bm{k}^T\bm{x}}.$$
    We now give the convergence of these two operators, which will be useful to the theoretical analysis of PM for QEOs in the next subsection.
	\begin{lemma}\label{lem-3}
		Let $r \geq s$, then for any $\mathcal{U} \in \mathcal{H}^{r}(\mathbb{T}^{n})$, 
		$$\|\mathcal{U}-\mathcal{P}_{N}\mathcal{U}\|_{s} \lesssim (1+N^{2})^{\frac{s-r}{2}}\|\mathcal{U}\|_{H^{r}(\mathbb{T}^{n})}.$$
	\end{lemma}
	\begin{proof}
    \begin{align*}
        \|\mathcal{U}-\mathcal{P}_{N}\mathcal{U}\|^{2}_{s}=&\sum_{\left|k_{i}\right| \geq N}(1+\left| \bm{Pk} \right|^{2})^{s} \left|\hat{\mathcal{U}}_{\bm{Pk}}\right|^{2} \\
              \lesssim&\sum_{\left|k_{i}\right| \geq N}(1+\left|\bm{k}\right|^{2})^{s-r+r}\left|\hat{\mathcal{U}}_{\bm{Pk}}\right|^{2}\\
        \lesssim&\sum_{\left|k_{i}\right| \geq N}(1+\left|\bm{k}\right|^{2})^{s-r}(1+\left| \bm{k} \right|^{2})^{r} \left|\hat{\mathcal{U}}_{\bm{Pk}}\right|^{2} \\
        \lesssim&(1+N^{2})^{s-r}\|\mathcal{U}\|_{\mathcal{H}^{r}(\mathbb{T}^{n})}^{2}.
    \end{align*}
	\end{proof}
    \begin{lemma}[\cite{mercier1989}, Theorem 8.1]\label{thm-2} For any $\mathcal{U} \in \mathcal{H}^{s}(\mathbb{T}^{n})$, 
		$$\|\mathcal{U}-\mathcal{I}_{N}\mathcal{U}\|_{\mathcal{L}^{2}(\mathbb{T}^{n})}\lesssim N^{-s/2}\|\mathcal{U}\|_{\mathcal{H}^{s}(\mathbb{T}^{n})}.$$
	\end{lemma}
	\begin{lemma}\label{lem-07281}
        Let $r \geq s$, then for any $\mathcal{U} \in \mathcal{H}^{r}(\mathbb{T}^{n})$, 
		$$\|\mathcal{U}-\mathcal{I}_{N}\mathcal{U}\|_{s} \lesssim (1+N^{2})^{\frac{s-r}{2}}\|\mathcal{U}\|_{\mathcal{H}^{r}(\mathbb{T}^{n})}.$$
	\end{lemma}
	\begin{proof}
		Noting that $\mathcal{I}_{N}$ leaves $V^{N}$ invariant, we have
		$$\mathcal{I}_{N}\mathcal{P}_{N} = \mathcal{P}_{N},$$
		and
		$$\mathcal{U}-\mathcal{I}_{N}\mathcal{U} = \mathcal{U}-\mathcal{P}_{N}\mathcal{U}+\mathcal{I}_{N}(\mathcal{P}_{N}-I)\mathcal{U}.$$
		Therefore, by setting $\mathcal{V} = (I-\mathcal{P}_{N})\mathcal{U}$, we have
		$$\|\mathcal{U}-\mathcal{I}_{N}\mathcal{U}\|_{s} \leq \|\mathcal{U}-\mathcal{P}_{N}\mathcal{U}\|_{s}+\|\mathcal{I}_{N}\mathcal{V}\|_{s}.$$ 
		The first term can be directly derived from \Cref{lem-3}. We now turn to the second one.
		
		Since $\mathcal{I}_{N}\mathcal{V} \in V^{N}$, we apply the inverse inequality to obtain
\begin{equation}
\begin{aligned}
		\|\mathcal{I}_{N}\mathcal{V}\|_{s}^{2}&= \sum_{\left|\bm{k}\right| \leq N}(1+\left|\bm{Pk}\right|^{2})^{s}\left|\hat{\mathcal{V}}_{\bm{Pk}}\right|^{2}\\
&\leq C(1+N^{2})^{s}\sum_{\left|\bm{k}\right| \leq N}(1+\left|\bm{Pk}\right|^{2})^{0}\left|\hat{\mathcal{V}}_{\bm{Pk}}\right|^{2}\\
 &= C(1+N^{2})^{s}\|\mathcal{I}_{N}\mathcal{V}\|_{0}^{2},\nonumber
\end{aligned}
\end{equation}
From \Cref{thm-2}, we have
	\begin{align*}
		\|\mathcal{I}_{N}\mathcal{V}\|_{0} &\leq \|\mathcal{U}-\mathcal{I}_{N}\mathcal{U}\|_{0}+\|\mathcal{U}-\mathcal{P}_{N}\mathcal{U}\|_{0} \\
		&\leq C(1+N^{2})^{-\frac{r}{2}}\|\mathcal{U}\|_{\mathcal{H}^{r}(\mathbb{T}^{n})}.
	\end{align*}
	\end{proof}

    \begin{lemma}\label{thm-11261}
        Consider the quasiperiodic elliptic equation $\mathcal{L}u=f$. 
		If the solution $u \in \mathcal{H}^{t}_{\rm QP}(\mathbb{R}^{d}) $ and its parent function $\mathcal{U} \in \mathcal{H}^{s}(\mathbb{T}^{n}),~s \leq t$, then the error estimate of PM is 
		$$\|\mathcal{U}-\mathcal{U}_{N}\|_{1} \lesssim (1+N^{2})^{-\frac{s-1}{2}}\|\mathcal{U}\| _{\mathcal{H}^{s}(\mathbb{T}^{n})}.$$
	\end{lemma}
	\begin{proof}
		We need to estimate
		$$\|\mathcal{U}-\mathcal{U}_{N}\|_{1} \leq \|\mathcal{U}-\mathcal{P}_{N}\mathcal{U}\|_{1}+\|\mathcal{P}_{N}\mathcal{U}-\mathcal{U}_{N}\|_{1}.$$ 
		For the second term of the error estimate, from the coercivity of the bilinear form in the numerical scheme
		\begin{equation}
			\begin{aligned}
				C\|\mathcal{P}_{N}\mathcal{U}-\mathcal{U}_{N}\|^{2}_{1} \leq&a_{p}(\mathcal{P}_{N}\mathcal{U}-\mathcal{U}_{N},\mathcal{P}_{N}\mathcal{U}-\mathcal{U}_{N})\\
				=&a_{p}({P}_{N}\mathcal{U},{P}_{N}\mathcal{U}-\mathcal{U}_{N})-a({P}_{N}\mathcal{U},{P}_{N}\mathcal{U}-\mathcal{U}_{N})\\
				&+a({P}_{N}\mathcal{U}-\mathcal{U},{P}_{N}\mathcal{U}-\mathcal{U}_{N})+(\mathcal{F}-\mathcal{P}_{N}\mathcal{F},{P}_{N}\mathcal{U}-\mathcal{U}_{N}). \nonumber
			\end{aligned}
		\end{equation}
The first two terms are given by
\begin{equation}
  \begin{aligned}
     &\left(\mathcal{I}_{N}(\mathcal{A}\bm{P} \nabla \mathcal{P}_{N}\mathcal{U}),\bm{P} \nabla(\mathcal{P}_{N}\mathcal{U}-\mathcal{U}_{N})\right)_{N}-\left(\mathcal{A}\bm{P} \nabla \mathcal{P}_{N}\mathcal{U},\bm{P} \nabla(\mathcal{P}_{N}\mathcal{U}-\mathcal{U}_{N})\right)\\
=&\left(\mathcal{I}_{N}(\mathcal{A}\bm{P} \nabla \mathcal{P}_{N}\mathcal{U})-\mathcal{A}\bm{P} \nabla \mathcal{P}_{N}\mathcal{U},\bm{P} \nabla(\mathcal{P}_{N}\mathcal{U}-\mathcal{U}_{N})\right),\nonumber
  \end{aligned}
\end{equation}
		Using \Cref{lem-07281}, the above term can be bounded by the interpolated projection inequality as
\begin{equation}
	\begin{aligned}
		&\|\mathcal{I}_{N}(\mathcal{A}\bm{P} \nabla \mathcal{P}_{N}\mathcal{U})-\mathcal{A}\bm{P} \nabla \mathcal{P}_{N}\mathcal{U}\|_{0} \|\mathcal{P}_{N}\mathcal{U}-\mathcal{U}_{N}\|_{1}\\
        \leq& C(1+N^{2})^{-\frac{s-1}{2}}\|\mathcal{A}\bm{P} \nabla \mathcal{P}_{N}\mathcal{U}\| _{s-1}\|\mathcal{P}_{N}\mathcal{U}-\mathcal{U}_{N}\|_{1}\\
		\leq&C(1+N^{2})^{-\frac{s-1}{2}}\|\bm{P}\nabla \mathcal{U}\| _{s}\|\mathcal{P}_{N}\mathcal{U}-\mathcal{U}_{N}\|_{1} \\
		\leq&C(1+N^{2})^{-\frac{s-1}{2}}\|\mathcal{U}\| _{H^{s}(\mathbb{T}^{n})}\|\mathcal{P}_{N}\mathcal{U}-\mathcal{U}_{N}\|_{1}. \nonumber
	\end{aligned}
\end{equation}
		The third term is given by \Cref{lem-3}
		\begin{equation}
			\begin{aligned}
				a({P}_{N}\mathcal{U}-\mathcal{U},{P}_{N}\mathcal{U}-\mathcal{U}_{N}) \leq& \|\mathcal{P}_{N}\mathcal{U}-\mathcal{U}_{N}\|_{1} \|\mathcal{P}_{N}\mathcal{U}-\mathcal{U}\|_{1} \\
				\leq& C(1+N^{2})^{-\frac{s-1}{2}}\|\mathcal{U}\| _{s}\|\mathcal{P}_{N}\mathcal{U}-\mathcal{U}_{N}\|_{1}.\nonumber
			\end{aligned}
		\end{equation}
		The fourth term follows from the orthogonality of the Fourier basis as
			\begin{align*}
				(\mathcal{F}-\mathcal{P}_{N}\mathcal{F},\mathcal{P}_{N}\mathcal{U}-\mathcal{U}_{N}) 
				\leq&C\|\mathcal{F}-\mathcal{P}_{N}\mathcal{F}\|_{-1}\|\mathcal{P}_{N}\mathcal{U}-\mathcal{U}_{N}\|_{1} \\
				\leq& C(1+N^{2})^{-\frac{s-2+1}{2}}\|\mathcal{F}\| _{s-2}\|\mathcal{P}_{N}\mathcal{U}-\mathcal{U}_{N}\|_{1}.
			\end{align*}
	\end{proof}

	\subsection{Convergence analysis}
    \label{subsec:con}
In this subsection, we analyze the convergence of our method for eigenvalues (see \Cref{thm:value}) and the corresponding eigenfunctions (see \Cref{thm:function}). 
    
    Before delving into the main results, we introduce the operator norm and the distance between subspaces. 
    For a bounded linear operator $B$, we define its subspace operator norm
$$\|B\|_{N} = \sup_{\mathcal{V}_{N} \in V^{N}}\frac{\|B\mathcal{V}_{N}\|}{\|\mathcal{V}_{N}\|}.$$
For a Banach space $V$, let $Y,Z$ be two closed subspaces. The distance between two subspaces is
	$$\delta (x,Z)=\inf_{y \in Z}\|x-y\|_{V},~\delta (Y,Z)=\sup_{x \in Y}\delta(x,Z),$$
	$$\hat{\delta}(Y,Z)={\rm max}\{\delta (Y,Z),\delta (Z,Y)\}.$$
	
	Denote $V=\mathcal{H}^{1}_{\bm{P}}(\mathbb{T}^{n})$ and $V^{N}={\rm span}\{e^{i\bm{k}^T\bm{x}}:\bm{k} \in K^{n}_{N},~\bm{x} \in \mathbb{T}^{n}\}$. 
	Since $a(\mathcal{W},\mathcal{V})$ is coercive, it satisfies the Lax-Milgram lemma, and \eqref{eq:source} and \eqref{eq:source_discrete} are well-posed. 
	Then, we can define the solution operators $K:V \to V$ and $K_{N}:V^{N} \to V^{N}$ by
	\begin{equation}\label{eqn-06171}
		\begin{aligned}
			a(K\mathcal{F},\mathcal{V})&=(\mathcal{F},\mathcal{V}),\qquad \forall\,\mathcal{V} \in V, \\
			a_{p}(K_{N}\mathcal{F}_{N},\mathcal{V}_{N})&=(\mathcal{F}_{N},\mathcal{V}_{N})_{N},\qquad \forall\,\mathcal{V}_{N} \in V^{N}. 
		\end{aligned}
	\end{equation}
	And there holds
	\begin{equation*}
		\|K\mathcal{F}\|_V\lesssim\|\mathcal{F}\|,~~\|K_N\mathcal{F}\|_{V_N}\lesssim\|\mathcal{F}\|,\quad \forall\,\mathcal{F}\in\mathcal{L}^2(\mathbb{T}^n).
	\end{equation*}

	Let $G_N:V\to V_N$ be the Ritz projection as follows
	\begin{equation*}
		a(G_N\mathcal{U},\mathcal{V})=a(\mathcal{U},\mathcal{V}),\qquad \forall \mathcal{V}\in V_N.
	\end{equation*}
	Then $K_N=G_N\circ K$ and there hold
	\begin{equation*}
		\lim_{N\to\infty}\|K-K_N\|_V=0.
	\end{equation*}

	\begin{lemma}\label{lem-11121}
		$K,K_{N}$ are self-adjoint.
	\end{lemma}
	\begin{proof}
    From the definitions of operators $K$, we have
		$$(K \mathcal{U},\mathcal{V})_{V}=a(K \mathcal{U},\mathcal{V})=(\mathcal{U},\mathcal{V})=a(\mathcal{U},K \mathcal{V})=(\mathcal{U},K\mathcal{V})_{V}.$$
		The conclusions regarding $K_N$ can be proven similarly to those for $K$.
	\end{proof}
    \noindent The original eigenvalue problem is equivalent to analyzing the nonzero spectrum of $K$ and $K_{N}$ and their associated eigenfunctions.
	Specifically, suppose $\mu$ is a nonzero eigenvalue of $K$ and $\mathcal{U}$ is an eigenfunction corresponding to $\mu$ of $K$, then
	$$a(\mathcal{U},\mathcal{V}) = \frac{1}{\mu}a(K\mathcal{U},\mathcal{V}) = \frac{1}{\mu}(\mathcal{U},\mathcal{V}),\qquad \forall\,\mathcal{V} \in V.$$
	Hence, $(\mu^{-1},\mathcal{U})$ is a solution of (\ref{eqn-5}). Similarly, if $(\gamma,\mathcal{U})$ is a solution of (\ref{eqn-5}), then $(\gamma^{-1},\mathcal{U})$ is an eigenpair of $K$.

    Next, we recall the resolvent set and spectrum. We denote the resolvent set by $\rho(K)$ and the spectrum by $\sigma(K)$, \textit{i.e.}, $\sigma(K) = \mathbb{C}\backslash \rho(K)$. For any $z \in \rho(K)$, $R_{z}(K)=(zI-K)^{-1}$ denotes the resolvent operator. Let $\mu$ be a nonzero eigenvalue of $K$ with algebraic or geometric multiplicity $m$. Let $\Gamma_{\mu}$ be a circle centered at $\mu$ that lies entirely in $\rho(K)$ and encloses no other points of $\sigma(K)$. The corresponding spectral projections $E_\mu:V \to V$ and $E_{\mu,N}:V^{N} \to V^{N}$ are defined by
	$$E_{\mu}=\frac{1}{2 \pi i}\int_{\Gamma_{\mu}}R_{z}(K){\rm d}z,\quad E_{\mu,N}=\frac{1}{2 \pi i}\int_{\Gamma_{\mu}}R_{z}(K_{N}){\rm d}z.$$
	It is known that the range $R(E_{\mu})$ of $E_{\mu}$ is the eigenspace corresponding to the eigenvalue $\mu$.

    From the spectral theory for non-compact operators in \cite{descloux1978}, when $K$ is not compact, $K_N$ can still provide a compact approximation to $K$ provided that the following two conditions hold:
\begin{itemize}
   \item $\textbf{Cond 1}: \lim\limits_{N\to \infty}\|K-K_{N}\|_{N} = 0$;\\
   \item $\textbf{Cond 2}: \lim\limits_{N \to \infty}\delta(\mathcal{V},V^{N})=0,\quad\forall\,\mathcal{V} \in V.$
\end{itemize}
Actually, since $\|K-K_{N}\|_{N}$ is the PM error for the elliptic equation, the estimate in \Cref{thm-11261} guarantees that $\textbf{Cond 1}$ holds.
Meanwhile, for $\textbf{Cond 2}$, by
$$\delta(\mathcal{V},V^{N}) = \inf_{\mathcal{V}_{N} \in V^{N}}\|\mathcal{V}-\mathcal{V}_{N}\|_{1} \leq \|\mathcal{V}-\mathcal{P}_{N}\mathcal{V}\|_1,$$
we obtain convergence immediately from the error estimate for the truncated operator $\mathcal{P}_{N}$ in \Cref{lem-3}. Hence, we have the following spectral convergence of $K_N$ to $K$.

\begin{theorem}[\cite{descloux1978}, Theorem 6]\label{thm:value}
    For the solution operator $K$ of the eigenproblem \eqref{eqn-5} and the corresponding discrete operator $K_N$, we have
	$$\lim\limits_{N \to \infty}\delta(\gamma,\sigma(K_{N})) = 0,\qquad \forall\,\gamma \in \sigma(K).$$
\end{theorem}

We now state two important results that follow from $\textbf{Cond 1}$. 
\begin{corollary}\label{cor-06231}
   When the property $\textbf{Cond 1}$ holds, there exists a constant $C$, independent of $N$, such that for $N$ sufficiently large,
    $$\|R_{z}(K_{N})\|_{N} \leq C,\qquad \forall\,z \in \rho(K).$$
\end{corollary}
\begin{proof}
   There exists $C>0$ such that
$$\|(zI-K)\mathcal{U}\| \geq 2C\|\mathcal{U}\|,\quad \forall\,\mathcal{U} \in V,~ z \in \rho(K).$$
By $\textbf{Cond 1},$ for $N$ large enough, we have 
$$\|(K-K_{N})\mathcal{U}\| \leq C\|\mathcal{U}\|,\quad \forall\,\mathcal{U} \in V^{N}.$$ 
Then we obtain for $\mathcal{U} \in V^{N},z \in \rho(K)$
$$\|(zI-K_{N})\mathcal{U}\| \geq \|(zI-K)\mathcal{U}\|-\|(K-K_{N})\mathcal{U}\| \geq C\|\mathcal{U}\|,$$
since $K_{N}$ is finite dimensional. This result also proves the existence of $R_{z}(K_{N})$.
\end{proof}
\begin{corollary}[\cite{descloux1978}, Theorem 1]\label{cor:des-thm1}
  When the property $\textbf{Cond 1}$ holds and $\Omega$ is an open set containing $\sigma(K)$, there exists $N_{0} > 0$ such that $\sigma(K_{N}) \subset \Omega,~ \forall\,N > N_{0}$.
\end{corollary}
\noindent An immediate implication of \Cref{cor:des-thm1} is that the proposed method does not generate spurious modes with eigenvalues mingled among those of physical relevance. In particular, such spectral pollution could, in principle, arise from the infinite-dimensional spectrum of the solution operator $K$.

\begin{theorem}\label{thm:function}
  For the eigenspace $E(V)$ of the problem \eqref{eqn-5} and the corresponding discrete eigenspace $E_N(V_N)$, we have
$$\lim\limits_{N \to \infty}\hat{\delta}(E(V),E_{N}(V^{N})) = 0.$$
\end{theorem}
\begin{proof}
  For $N \in \mathbb{N}^+$,
   \begin{equation}\label{eqn-07051}
      \begin{aligned}
          \|E-E_{N}\|_{N} &\leq \frac{1}{2 \pi}\int_{\Gamma}\|R_{z}(K)-R_{z}(K_{N})\|_{N}|{\rm d}z|\\
&=\frac{1}{2 \pi}\int_{\Gamma}\|R_{z}(K)(K-K_{N})R_{z}(K_{N})\|_{N}|{\rm d}z|\\
& \leq \frac{1}{2 \pi}\int_{\Gamma}\|R_{z}(K)\|_{N}\|K-K_{N}\|_{N}\|R_{z}(K_{N})\|_{N}|{\rm d}z|.
      \end{aligned}
   \end{equation}
From $\textbf{Cond 1}$ and \Cref{cor-06231}, we get 
$$\lim\limits_{N \to \infty}\|E-E_{N}\|_{N} = 0,$$
which means that
\begin{equation}\label{eqn-07052}
\lim\limits_{N \to \infty}\delta(E_{N}(V^{N}),E(V)) = 0.
\end{equation}

Furthermore, for all $\mathcal{V} \in E(V)$, by $\textbf{Cond 2}$ there exists $\mathcal{V}_{N} \in V^{N}$ such that $\lim\limits_{N \to \infty}\|\mathcal{V}-\mathcal{V}_{N}\| = 0$. Then 
\begin{equation}
  \begin{aligned}
      \|\mathcal{V}-E_{N}\mathcal{V}_{N}\| &= \|E\mathcal{V}-E_{N}\mathcal{V}_{N}\|\\
&\leq \|E(\mathcal{V}-\mathcal{V}_{N})\|+\|(E-E_{N})\mathcal{V}_{N}\|\\
&\leq \|E\| ~\|\mathcal{V}-\mathcal{V}_{N}\|+\|E-E_{N}\|_{N}\|\mathcal{V}_{N}\|.\nonumber
  \end{aligned}
\end{equation}
By (\ref{eqn-07051}) and the continuity of $E$, we obtain 
\begin{equation}\label{eqn-07053}
\lim\limits_{N \to \infty}\delta(\mathcal{V},E_{N}(V^{N})) = 0.
\end{equation}
Let $m$ and $m_{N}$ be the dimensions of $E(V)$ and of $E_{N}(V^{N})$, respectively. \Cref{eqn-07053} shows that if $m = \infty$ then $\lim\limits_{N \to \infty}m_{N} = \infty$. If $m<\infty$ then $$\lim\limits_{N \to \infty}\delta(E(V),E_{N}(V^{N})) = 0.$$ Combining this with (\ref{eqn-07052}), we obtain $m = m_{N}$ for sufficiently large $N$ and
$$\lim\limits_{N \to \infty}\hat{\delta}(E(V),E_{N}(V^{N})) = 0.$$
\end{proof}

\subsection{Error estimates}
Let $\gamma \in \mathbb{C}$ be an isolated eigenvalue of the solution operator $K$ with finite algebraic multiplicity $m$. 
The coercivity of the bilinear form $a(\cdot,\cdot)$ implies $\gamma \neq 0$. 
Consequently, there exists a closed disk $\mathcal{B}_\gamma \subset \mathbb{C} \backslash \{0\}$ centered at $\gamma$ such that $\mathcal{B}_\gamma \cap \sigma(K)=\{\gamma\}$. Let $\{\gamma_{i,N}\}^{m_{N}}_{i=1}$ denote eigenvalues of the discrete solution operator $K_{N}$ contained in $\mathcal{B}_\gamma$. Subsection 3.3 establishes that
\begin{itemize}
   \item $~ m_N = m$ for sufficiently large $N$;
   \item $\lim\limits_{N \to \infty}\gamma_{i,N} = \gamma,~ i = 1,\cdots,m$.
\end{itemize}
Building upon these results, we now derive error estimates for the eigenvalues $\gamma_{i,N}$ and their associated eigenfunctions. 

	\begin{theorem}
		Assume that $\gamma$ is an eigenvalue of \eqref{eqn-5} with algebraic multiplicity $m$. Then, for $N$ sufficiently large, there are $m$ numerical eigenvalues $\{\gamma_{j,N}\}^{m}_{j=1}$ such that
		$$\left| \gamma - \gamma_{j,N} \right| \leq C(1+N^{2})^{-(s-1)},\qquad j=1,\cdots,m.$$
	\end{theorem}
	\begin{proof}
		Assume that $\{\mathcal{U}_{j}\}^{m}_{j=1}$ and $\{\mathcal{U}_{j,N}\}^{m}_{j=1}$ are eigenfunctions corresponding to $\gamma$ and $\{\gamma_{j,N}\}^{m}_{j=1}$, respectively. 
        For any $\mathcal{V}_{N} \in V^{N}$,
			\begin{equation}
					\begin{aligned}
							&\gamma_{N}(\mathcal{U}_{N}-\mathcal{V}_{N},\mathcal{U}_{N}-\mathcal{V}_{N})_{N}+\gamma_{N}((\mathcal{U},\mathcal{U})-(\mathcal{V}_{N},\mathcal{V}_{N})_{N}) \\
							=&\gamma_{N}(\mathcal{U}_{N}-\mathcal{V}_{N},\mathcal{U}_{N}-\mathcal{V}_{N})_{N}+\gamma_{N}((\mathcal{U}_{N},\mathcal{U}_{N})_{N}-(\mathcal{V}_{N},\mathcal{V}_{N})_{N}) \\
							=&2 \gamma_{N}(\mathcal{U}_{N},\mathcal{U}_{N})_{N}-2\gamma_{N}(\mathcal{U}_{N},\mathcal{V}_{N})_{N} \\
							=&2 \gamma_{N}(\mathcal{U}_{N},\mathcal{U}_{N})_{N}-2a_{p}(\mathcal{U}_{N},\mathcal{V}_{N}), \nonumber
						\end{aligned}
				\end{equation}
			then
	\begin{equation}
		\begin{aligned}
			\gamma - \gamma_{N}=&\gamma (\mathcal{U},\mathcal{U})+\gamma_{N}(\mathcal{U}_{N},\mathcal{U}_{N})_{N}-2\gamma_{N}(\mathcal{U}_{N},\mathcal{U}_{N})_{N} \\
			=&a(\mathcal{U},\mathcal{U})+a_{p}(\mathcal{U}_{N},\mathcal{U}_{N})-2a_{p}(\mathcal{U}_{N},\mathcal{V}_{N}) \\
			&-\gamma_{N}(\mathcal{U}_{N}-\mathcal{V}_{N},\mathcal{U}_{N}-\mathcal{V}_{N})_{N}-\gamma_{N}((\mathcal{U},\mathcal{U})-(\mathcal{V}_{N},\mathcal{V}_{N})_{N})\\
			=&a(\mathcal{U},\mathcal{U})-a_{p}(\mathcal{V}_{N},\mathcal{V}_{N})+a_{p}(\mathcal{U}_{N}-\mathcal{V}_{N},\mathcal{U}_{N}-\mathcal{V}_{N}) \\
			&-\gamma_{N}(\mathcal{U}_{N}-\mathcal{V}_{N},\mathcal{U}_{N}-\mathcal{V}_{N})_{N}-\gamma_{N}((\mathcal{U},\mathcal{U})-(\mathcal{V}_{N},\mathcal{V}_{N})_{N}).\nonumber
		\end{aligned}
	\end{equation}
	Let $\mathcal{V}_{N}=\mathcal{P}_{N}\mathcal{U}$, then we have
    \begin{equation*}\label{eq:lambda}
			\begin{aligned}
				\gamma - \gamma_{N}=&a(\mathcal{U},\mathcal{U})-a_{p}(\mathcal{P}_{N}\mathcal{U},\mathcal{P}_{N}\mathcal{U})+a_{p}(\mathcal{U}_{N}-\mathcal{P}_{N}\mathcal{U},\mathcal{U}_{N}-\mathcal{P}_{N}\mathcal{U}) \\
				-&\gamma_{N}(\mathcal{U}_{N}-\mathcal{P}_{N}\mathcal{U},\mathcal{U}_{N}-\mathcal{P}_{N}\mathcal{U})_{N}-\gamma_{N}\left((\mathcal{U},\mathcal{U})-(\mathcal{P}_{N}\mathcal{U},\mathcal{P}_{N}\mathcal{U})_{N}\right). 
			\end{aligned}
		\end{equation*}      
        The eigenvalue error estimation reduces to bounding six terms in the equation above.
		
		For the first two terms, by the orthogonality of the truncated projection and \Cref{lem-3}, we have
		\begin{equation}
			\begin{aligned}
				&\left|(\mathcal{A}\bm{P} \nabla \mathcal{U}_{j},\bm{P}\nabla \mathcal{U}_{j})-(\mathcal{A}\bm{P} \nabla \mathcal{P}_{N}\mathcal{U}_{j},\bm{P}\nabla \mathcal{P}_{N}\mathcal{U}_{j})_{N}\right| \\
				\leq& C(\| \bm{P}\nabla \mathcal{U}_{j}\|^{2}-\|  \bm{P}\mathcal{P}_{N} \nabla \mathcal{U}_{j}\|^{2}) \\
				\leq& C\|\nabla \mathcal{U}_{j}-\mathcal{P}_{N} \nabla \mathcal{U}_{j}\|^{2} \\
				\leq& C(1+N^{2})^{-(s-1)}\|\nabla \mathcal{U}_{j}\|^{2}_{s-1}\\
				\leq& C(1+N^{2})^{-(s-1)}\| \mathcal{U}_{j}\|^{2}_{s}.\nonumber
			\end{aligned}
		\end{equation}
		In addition,
		\begin{equation}
			\begin{aligned}
				&\left|a_{p}(\mathcal{U}_{j,N}-\mathcal{P}_{N}\mathcal{U}_{j},\mathcal{U}_{j,N}-\mathcal{P}_{N}\mathcal{U}_{j})\right|+\left|\gamma_{j,N}(\mathcal{U}_{j,N}-\mathcal{P}_{N}\mathcal{U}_{j},\mathcal{U}_{j,N}-\mathcal{P}_{N}\mathcal{U}_{j})_{N}\right|\\
				\leq &\|\mathcal{U}_{j,N}-\mathcal{P}_{N}\mathcal{U}_{j}\|_{1}^{2}\\
				\leq &\|\mathcal{U}_{j,N}-\mathcal{U}_{j}\|_{1}^{2}+\|\mathcal{U}_{j}-\mathcal{P}_{N}\mathcal{U}_{j}\|_{1}^{2} \\
				\leq &C(1+N^{2})^{-(s-1)}.\nonumber
			\end{aligned}
		\end{equation}
		Similarly, a comparable estimate can be obtained for the last two terms. 

        Therefore, we have
		\begin{equation}
			\begin{aligned}
				\left| \gamma - \gamma_{j,N} \right| 
				\leq&C(1+N^{2})^{-(s-1)}+\|\mathcal{U}_{j,N}-\mathcal{U}_{j}\|^{2}_{1}+\|\mathcal{U}_{j}-\mathcal{P}_{N}\mathcal{U}_{j}\|_{1}^{2} \\
				\leq&C(1+N^{2})^{-(s-1)},\nonumber
			\end{aligned}
		\end{equation}
which completes the proof.
	\end{proof}

Define the orthogonal projection operator $\Pi_{N}:V \to V^{N}$ by the variational relation 
\begin{equation}\label{eq:pro_ope}
    a(\Pi_{N}\mathcal{U}-\mathcal{U},\mathcal{V}) = 0,\qquad \forall\,\mathcal{V} \in V^{N}.
\end{equation}
This induces the discrete operator $K_{N} = \Pi_{N}K|_{V^{N}}$. Setting $T_{N} = \Pi_{N}K\Pi_{N}:V \to V$, we note that, except for zero, $T_{N}$ has the same spectrum and corresponding invariant subspaces as $K_{N}$. The spectral projection of $T_{N}$ relative to $\gamma_{1,N}, \cdots, \gamma_{m,N}$ is
\begin{equation}\label{eq:spe_pro_F}
    F_{\gamma,N}=\frac{1}{2 \pi i}\int_{\Gamma_{\mu}}R_{z}(T_{N}){\rm d}z.
\end{equation}
Now consider adjoint operators. The adjoint $K^{*}$ possesses an isolated eigenvalue $\bar{\gamma}$ with algebraic multiplicity $m$. The dual projection $\Pi_{N}^{*}$ satisfies
$$a(\mathcal{V},\Pi_{N}^{*}\mathcal{U}-\mathcal{U}) = 0,\qquad \forall\,\mathcal{V} \in V^{N}.$$ 
The spectral projections $E^{*}$ at $\bar{\gamma}$ and $F_{N}^{*}$ at $\bar{\gamma}_{1,N}, \cdots, \bar{\gamma}_{m,N}$ admit the representations
$$E^{*}_{\gamma}=\frac{1}{2 \pi i}\int_{\bar{\Gamma}_{\gamma}}R_{z}(K^{*}){\rm d}z,\quad F^{*}_{\gamma,N}=\frac{1}{2 \pi i}\int_{\bar{\Gamma}_{\gamma}}R_{z}(T^{*}_{N}){\rm d}z,$$
where $\bar{\Gamma}_{\gamma}$ denotes the complex conjugate of the contour $\Gamma_{\gamma}$.

\begin{lemma}
The projection operator $\Pi_{N}$ defined by \eqref{eq:pro_ope} is uniformly bounded with respect to $N$. Moreover,
$$\|\mathcal{U}-\Pi_{N}\mathcal{U}\| \lesssim \delta(\mathcal{U},V^{N}),\qquad \forall\,\mathcal{U} \in V.$$
The adjoint projection $\Pi_{N}^{*}$ satisfies the same bound.
\end{lemma}
\begin{proof}
   Fix $\mathcal{U} \in V$. From the definition of $\Pi_{N}$, let $\mathcal{V} = \Pi_{N}\mathcal{U}$, then $$a(\Pi_{N}\mathcal{U},\Pi_{N}\mathcal{U}) = a(\mathcal{U},\Pi_{N}\mathcal{U}).$$ 
   Then 
$$C_{0}\|\Pi_{N}\mathcal{U}\|^{2} \leq |a(\Pi_{N}\mathcal{U},\Pi_{N}\mathcal{U})| = |a(\mathcal{U},\Pi_{N}\mathcal{U})| \leq C_{1}\|\mathcal{U}\|\|\Pi_{N}\mathcal{U}\|.$$
Similarly, let $\mathcal{V} = \mathcal{U}-\Pi_{N}\mathcal{U}$, we have
$$a(\mathcal{U}-\Pi_{N}\mathcal{U},\mathcal{U}-\Pi_{N}\mathcal{U})^{\frac{1}{2}} = \inf_{\mathcal{U}_{N} \in V^{N}}a(\mathcal{U}-\mathcal{U}_{N},\mathcal{U}-\mathcal{U}_{N})^{\frac{1}{2}},$$
and
\begin{equation*}
C_{0}\|\mathcal{U}-\Pi_{N}\mathcal{U}\|^{2} \leq C_{1}\|\mathcal{U}-\mathcal{U}_{N}\|\|\mathcal{U}-\Pi_{N}\mathcal{U}\|,~ \forall\,\mathcal{U}_{N} \in V^{N}.
\end{equation*}
\end{proof}

\begin{theorem}[\cite{descloux19782}, Theorem 1]
  Assume that 
  $E_{\gamma}(V)$ is the eigenspace with respect to the eigenvalue $\gamma$ of \eqref{eqn-5}, and the corresponding eigenspace $F_{\gamma,N}(V)$ is defined by \eqref{eq:spe_pro_F}, 
  there exists
$$\hat{\delta}(F_{\gamma,N}(V),E_{\gamma}(V)) \leq \gamma_N,~ \hat{\delta}(F^{*}_{\gamma,N}(V),E^{*}_{\gamma}(V)) \leq \gamma_N^*,$$
where 
$$\gamma_{N} = \delta(E_{\gamma}(V),V^{N}),~ \gamma_{N}^{*} = \delta(E_{\gamma}^{*}(V),V^{N}).$$
\end{theorem}
\noindent Since $K$  is self-adjoint, all expressions above can be rigorously bounded via error estimates by \Cref{lem-3}. 
Then we have $$\hat{\delta}(F_{\gamma,N}(V),E_{\gamma}(V)) \lesssim (1+N^{2})^{\frac{1-s}{2}},$$
$$\hat{\delta}(F^{*}_{\gamma,N}(V),E^{*}_{\gamma}(V)) \lesssim (1+N^{2})^{\frac{1-s}{2}}.$$

\section{Numerical experiments}
\label{sec:exp}
In this section, we verify the accuracy and efficiency of our approach through numerical experiments, including a one-dimensional photonic quasicrystal and two- and three-dimensional QEOs.
All experiments are implemented in MATLAB R2022a on a laptop equipped with an Intel Core processor (2.60 GHz) and 8 GB of RAM.

To solve the discretized eigenvalue system \eqref{eq:lin_sys}, we use the eigs solver in MATLAB. 
As the discretization grid size $N$ increases, the linear system becomes increasingly ill-conditioned due to  {the stiffness of the discretized operators}, making it imperative to design an effective preconditioner. We construct a diagonal preconditioner by solving the following matrix minimization problem
\begin{equation}
   \begin{aligned}
     \bm{M} = \underset{\bm{D}\in \mathcal{D}}{\rm argmin}\|\bm{D}\bm{Q}-\bm{I}_{D}\|_{F},
   \end{aligned}
\end{equation}
where $\mathcal{D}$ denotes the set of all $N$-order diagonal matrices. 
As established in \cite{jiang2024}, the solution of the above optimization problem is
$$\bm{M} = {\rm diag}(q_{11}/\|\bm{e}_{1}^{T}\bm{Q}\|_{2}^{2},\cdots,q_{DD}/\|\bm{e}_{D}^{T}\bm{Q}\|_{2}^{2}).$$
To validate the diagonal preconditioner $\bm{M}$, Table \ref{tab.con} reports results for $d=1$, $n=2$ (see Example 4.1). The condition number drops from $O(10^{3})$ to single digits and remains stable as $N$ increases. 
In 2D and 3D tests, the condition numbers show a similar decrease, from $O(10^{4})$ and $O(10^{5})$ to single-digit values.

	\begin{table}[htbp]
        \footnotesize
        \centering
        \caption{Condition numbers of matrix $\bm{Q}$ and its corresponding preconditioned matrix $\bm{MQ}$.}
        \label{tab.con}
		\begin{tabular}{||c|c|c|c|c||} 
			\hline
			 $N$ & 8 & 16 & 32 & 64\\
			\hline
			\hline
			${\rm cond}(\bm{Q})$ & 1.8848e+03 & 1.2779e+05 & 5.1708e+05 & 3.8355e+07\\
			\hline
			${\rm cond}(\bm{MQ})$ & 4.16 & 4.40 & 4.76 & 4.77\\
			\hline
		\end{tabular}
\end{table}

In what follows, we present convergence results for eigenpairs of QEOs computed by our method, focusing on  {the ground and first excited states} which are first identified on the coarse grid. For each refined grid, we track the corresponding eigenpairs and quantify their errors and convergence rates to evaluate numerical performance. While the method exhibits consistent convergence behavior across the spectrum, we limit the presentation to  {these two states} for practical reasons.

\begin{example}
    We begin with a one-dimensional photonic quasicrystal.
	Here, the QEO problem can be derived from Maxwell's equations for a scalar magnetic-field component $ H(x) $ \cite{chan1998photonic}
    $$-\frac{d}{dx}\left(\frac{1}{\varepsilon(x)}\frac{d}{dx}H(x)\right) = k^{2}H(x),$$
    where $ k $ is the free-space wavenumber and $ \varepsilon(x) $ denotes the dielectric permittivity. The inverse of $ \varepsilon(x) $ is modeled as a quasiperiodic potential generated by two incommensurate photonic lattices
    $$\alpha_1(x)=\frac{1}{\varepsilon(x)} = \frac{1}{2}\left({\rm cos}(x)+{\rm cos}\left(\beta x\right)\right)+1,$$
    with $ \beta \in \mathbb{R} \backslash \mathbb{Q} $. Here, we set $ \beta = (\sqrt{5}-1)/2 $.
\end{example}

    The projection matrix is $\bm{P} = 2 \pi (1,(\sqrt{5}-1)/2)$. $\alpha_{1}(x)$ can be embedded into the 2D parent function $A_{1}(\bm{y}) = \frac{1}{2}({\rm cos}(y_{1})+{\rm cos}(y_{2}))+1, \bm{y} = (y_{1},y_{2})^{T}.$
	We take the numerical solution at $N=128$ as the reference.  {We first identify the ground and first excited states on the coarsest mesh ($N=8$). Then, on finer meshes, we track the corresponding eigenvalues by matching their eigenfunctions.} 
	\Cref{Fig.ex1} illustrates that the errors in both the eigenvalues and eigenfunctions decrease monotonically as $N$ increases from 8 to 32.

\begin{figure}[htbp] 
	\centering 
	\includegraphics[width=0.65\textwidth]{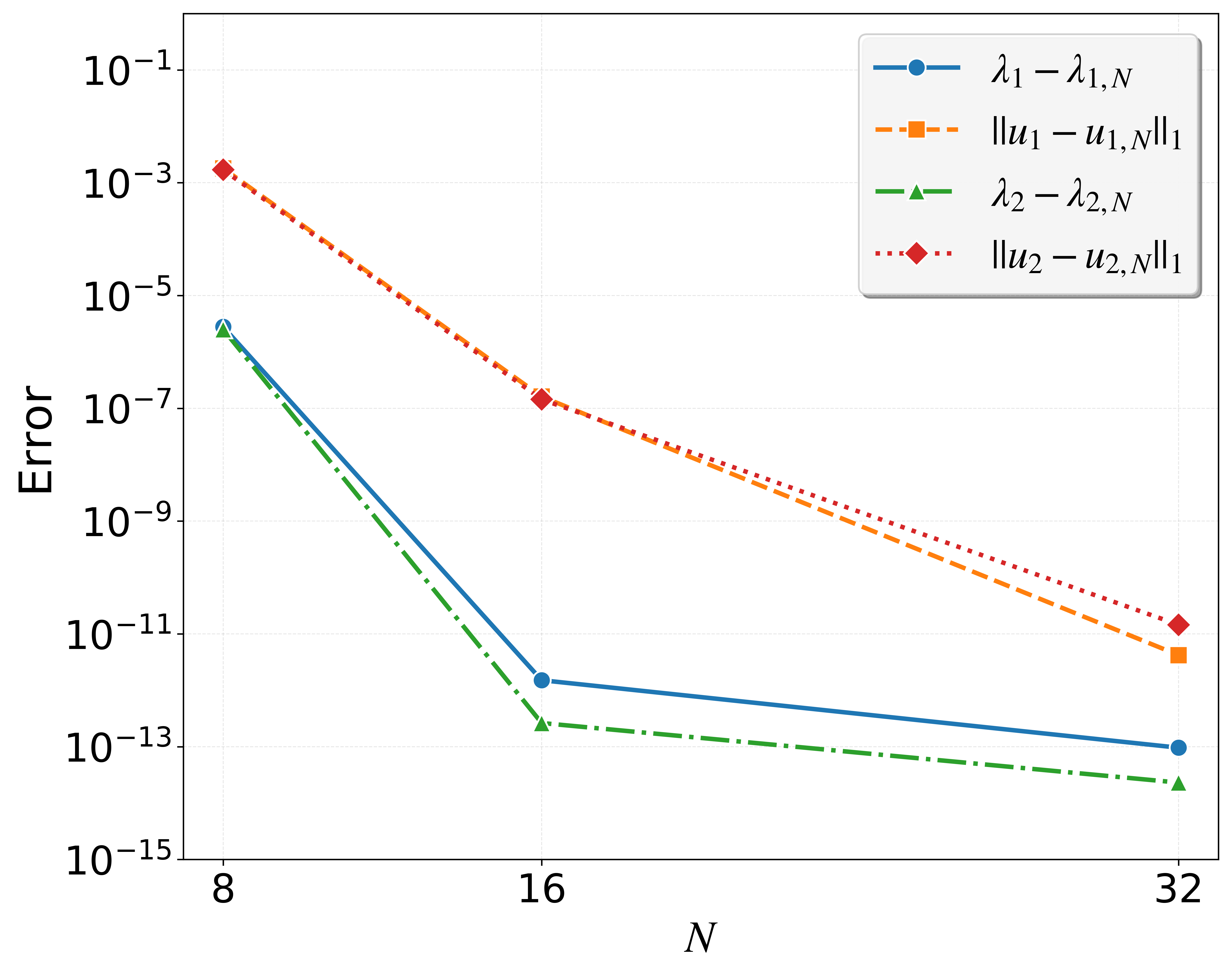} 
	\caption{Errors of the first two eigenpairs when using PM to solve Example 4.1.} 
	\label{Fig.ex1} 
\end{figure}

 \begin{figure*}[!htbp] 
 	\centering 
 	\includegraphics[width=1.0\textwidth]{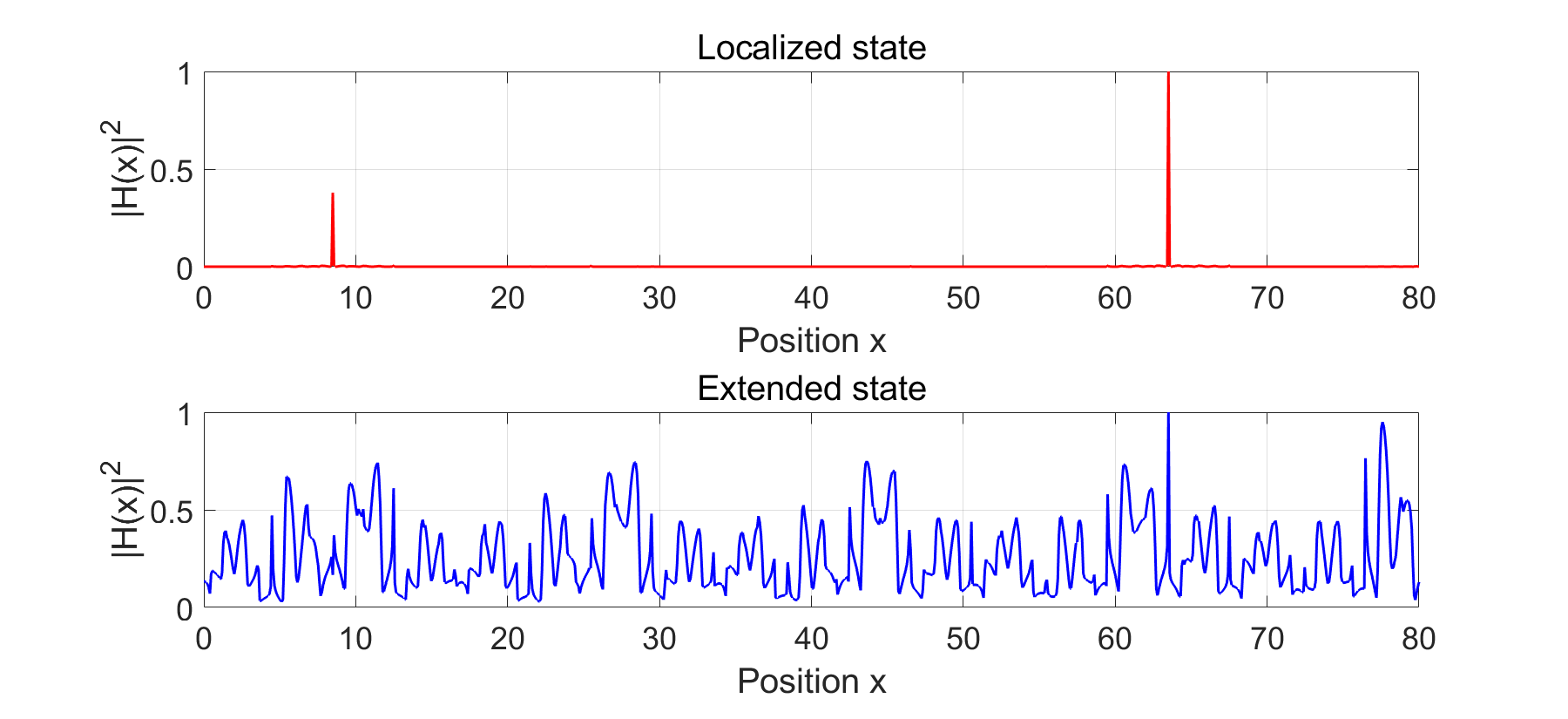} 
 	\caption{Localized and extended states of Example 4.1 solved by PM.} 
 	\label{Fig.sta} 
 \end{figure*}

Beyond convergence accuracy, \Cref{Fig.sta} shows two representative eigenstates computed by our method. The top subfigure illustrates a localized state with sharp, concentrated peaks (a double-peak pattern typical of localization~\cite{ge2019exponential}). In contrast, the bottom subfigure displays an extended state with oscillations distributed across the domain. These two eigenstates reflect distinct propagation behaviors in photonic quasicrystals induced by quasiperiodicity, consistent with experimental observations~\cite{chan1998photonic,vardeny2013optics}. 

$\textbf{Comparison with PAM.}$ 
To further evaluate the proposed method, we solve Example 4.1 with PAM and compare the results. PAM approximates the original coefficient $\alpha_1(x)$ by
$$\alpha_{1L}(x) = {\rm cos}(x)+{\rm cos}\left(\frac{\lfloor(\sqrt{5}-1)L\rfloor}{2L} x\right)+1,~ L \in \mathbb{Z}$$ 
where $\lfloor \cdot \rfloor$ denotes the floor function. Approximating the irrational $\beta$ by a rational number makes the potential periodic, so the problem reduces to a periodic system with period $2\pi L$ (with $L$ as a tunable parameter). This rational approximation introduces a Diophantine error, defined as
\begin{equation*}
     e_{d}:= \left|\beta- \frac{\lfloor(\sqrt{5}-1)L\rfloor}{2L}\right|.
\end{equation*} 
\Cref{Fig.main1} shows the Diophantine error $e_d$ versus the integer parameter $L$ for $0 \leq L \leq 500$. The error monotonically decreases below $10^{-2}$ only when $L$ is chosen from the best rational approximation sequence of $\beta$ (specifically, $L=17, 72, 144, 233, 305, 377$). 

\begin{figure}[htbp] 
	\centering 
	\includegraphics[width=0.75\textwidth]{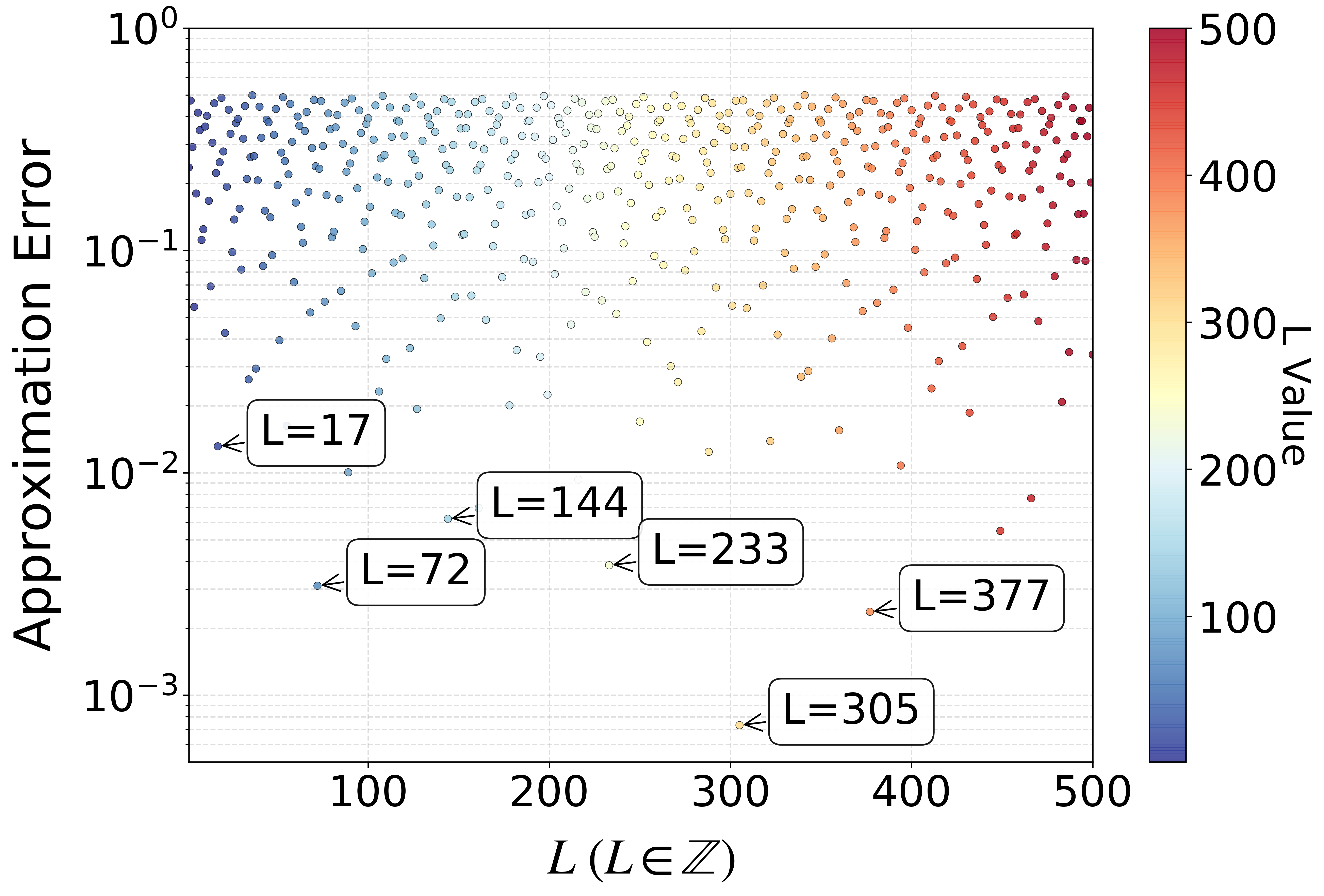} 
	\caption{Diophantine approximation error $e_{d}$ when using PAM to solve Example 4.1.} 
	\label{Fig.main1} 
\end{figure}

	\begin{table}[htbp]
        \footnotesize
        \caption{Errors of the first two eigenvalues when using PAM to solve Example 4.1.}
        \label{tab1}
		\begin{center}
			\begin{tabular}{||c|c|c|c|c|c|c||} 
				\hline
				$L$ & 17 & 72 & 144 & 233 & 305 & 377\\
				\hline
				\hline
				$e_{d}$ & 1.3156e-02 & 3.1056e-03 & 6.2112e-02 & 3.8358e-03 & 7.3314e-04 & 2.3752e-03\\
				\hline
				$\gamma_{1}-\gamma_{1,L}$ & 8.1600e-02 & 4.4000e-03 & 4.4000e-02 & 1.7000e-03 & 2.4649e-04 & 6.4513e-03\\
				\hline
				$\gamma_{2}-\gamma_{2,L}$ & 8.1600e-02 & 1.7800e-02 & 1.7800e-02 & 6.8000e-03 & 9.8976e-04 & 2.6000e-03\\
				\hline
			\end{tabular}
		\end{center}
	\end{table}

    \begin{figure}[htbp]
	\centering 
	\includegraphics[width=0.8\textwidth]{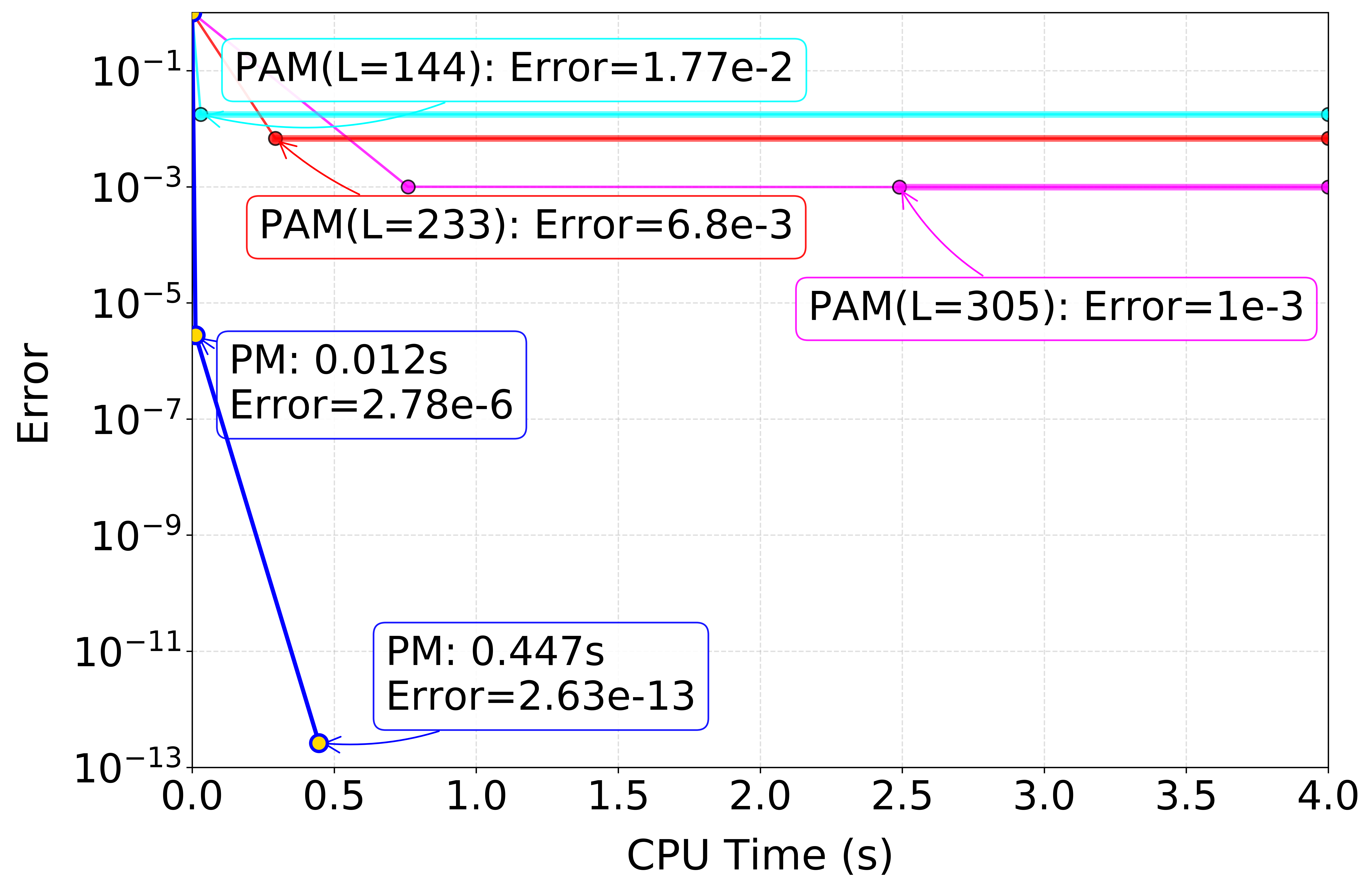} 
	\caption{Comparison of PM and PAM when solving Example 4.1.} 
	\label{Fig.PMPAM} 
    \end{figure} 

For PAM simulations, we use these optimal $L$ values and set $D = L \times N$ discrete points (with $N=16$ fixed) to ensure sufficient accuracy. 
 {We take the numerical solution computed by PM at $N=128$ as the reference.} 
\Cref{tab1} lists the numerical errors of tracked eigenvalues; the results confirm that these errors are primarily governed by the Diophantine error $e_d$~\cite{jiang2023}. This limitation makes PAM impractical in applications.

We further compare the efficiency of PM and PAM. \Cref{Fig.PMPAM} reports the relative error of the first eigenvalue and the corresponding CPU time. For quasiperiodic problems, PM is both more accurate and more efficient, achieving higher accuracy with substantially less CPU time.

\begin{example} 
Consider 2D QEO (\ref{eqn-4}) with the coefficient
	$$\alpha_{2}(\bm{x}) = {\rm cos}(\beta x_{1})+{\rm cos}(\beta x_{2})+{\rm cos}\left(\beta {\rm cos}(\theta)x_{1}+\beta {\rm sin}(\theta)x_{2}\right)+6.$$
	where $\bm{x} = (x_{1},x_{2})^{T},~ \beta \in \mathbb{R},~ \theta \in (0,2\pi)$. The parameter $\theta$ indicates the relative angle of rotation between the lattices.
\end{example}
	
	The projection matrix of $\alpha_{2}(\bm{x})$ is
	\begin{equation}
		\bm{P} = \beta \begin{pmatrix} 1 & 0 & {\rm cos}(\theta) \\ 0 & 1 & {\rm sin}(\theta) \end{pmatrix}, \nonumber
	\end{equation}
	and the corresponding parent function is
	$$A_{2}(\bm{y}) = {\rm cos}(y_{1})+{\rm cos}(y_{2})+{\rm cos}(y_{3})+6,~ \bm{y} = (y_{1},y_{2},y_{3})^{T}.$$

    To assess accuracy and convergence, we take the numerical solution at $N=64$ as the reference. \Cref{Fig.ex2} shows the errors of the first two eigenpairs as $N$ increases from 8 to 32. The blue and green curves represent the first and second eigenvalue errors, respectively, exhibiting exponential convergence and reaching the $10^{-15}$ level at $N=32$. The orange and red dashed curves show the corresponding eigenfunction errors, which also decay monotonically with $N$. These results demonstrate high-precision convergence for both eigenvalues and eigenfunctions, validating the accuracy and efficiency of the method for this 2D QEO problem.

\begin{figure}[htbp] 
	\centering 
	\includegraphics[width=0.65\textwidth]{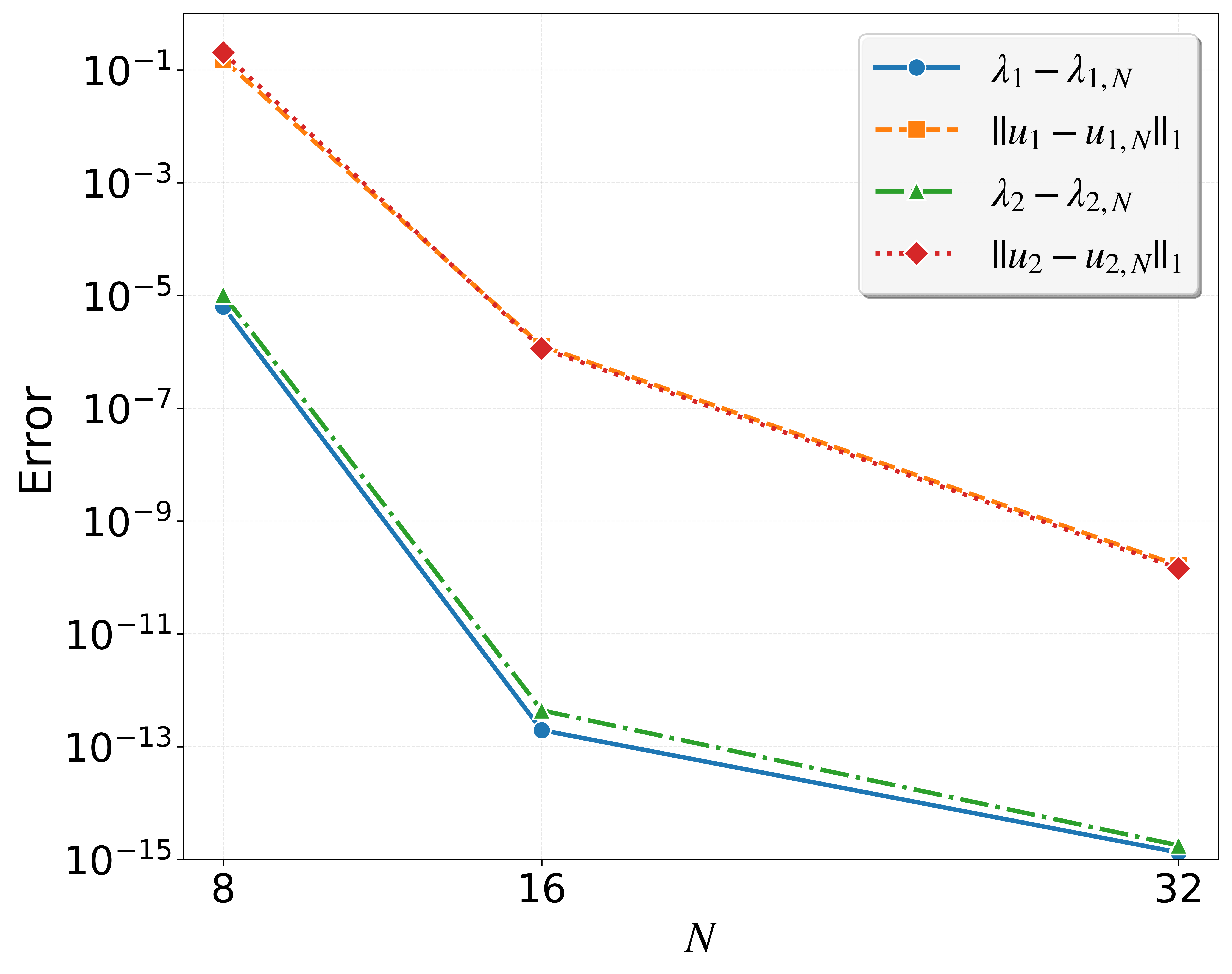} 
	\caption{Errors of the first two eigenpairs when using PM to solve Example 4.2 with $\beta = 0.5\pi,~\theta = 0.2 \pi$.} 
	\label{Fig.ex2} 
\end{figure}

\begin{example}
 Consider 3D QEO (\ref{eqn-4}) with the coefficient
	$$\alpha_{3}(\bm{x}) = {\rm cos}(2 \pi x_{1})+{\rm cos}(2 \pi x_{2})+{\rm cos}(2 \pi x_{3})+{\rm cos}(2 \pi \beta x_{3})+8.$$
where $\beta = \sqrt{5}-1$. 
\end{example}

The projection matrix is
	\begin{equation}
		\bm{P} = 2 \pi \begin{pmatrix} 1 & 0 & 0 & 0 \\ 0 & 1 & 0 & 0 \\0 & 0 & 1 & \beta \end{pmatrix}, \nonumber
	\end{equation}
and the corresponding parent function is
	$$A_{3}(\bm{y}) = {\rm cos}(y_{1})+{\rm cos}(y_{2})+{\rm cos}(y_{3})+{\rm cos}(y_{4})+8,~ \bm{y} = (y_{1},y_{2},y_{3},y_{4})^{T}.$$

	\begin{table}[htbp]
        \footnotesize
        \caption{Errors and orders of the first three eigenvalues when using PM to solve Example 4.3.}
        \label{tab2}
		\begin{center}
			\begin{tabular}{||c|c|c|c|c|c||} 
				\hline
				$N$ & 6 & 8 & 10 & 12 & 14\\
				\hline
				\hline
				$\gamma_{1}-\gamma_{1,N}$ & 4.4322e-05 & 4.2692e-07 & 3.8024e-09 & 3.2880e-11 & 2.7711e-13\\
				\hline
				$\gamma_{2}-\gamma_{2,N}$ & 5.1589e-05 & 4.6952e-07 & 4.0744e-09 & 3.4962e-11 & 2.9132e-13\\
				\hline
				$\gamma_{3}-\gamma_{3,N}$ & 0.0499 & 1.1096e-04 & 7.6371e-07 & 1.0143e-08 & 1.8659e-13\\
				\hline
			\end{tabular}
		\end{center}
	\end{table}

    As a reference, we use the numerical solution at $N=16$. \Cref{tab2} reports the absolute errors of the first three eigenvalues for $N=6,8,10,12,14$. The eigenvalue errors decrease exponentially as $N$ grows. These results corroborate the high-precision convergence of the method for this 3D QEO case.

\section{Conclusion and outlook}
\label{sec:con}
	This paper develops a PM-based algorithm for computing QEO eigenpairs by embedding quasiperiodic operators into a higher-dimensional periodic torus and solving the equivalent problem with Fourier discretization. To address the theoretical challenge posed by the non-compactness of QEOs in Besicovitch quasiperiodic function spaces, we construct a directional-derivative Hilbert space along irrational manifolds of the high-dimensional torus.
	By integrating spectral theory for non-compact operators into the Babu\v{s}ka--Osborn eigenproblem framework, we establish rigorous convergence analysis and error estimates for both eigenvalues and eigenfunctions.
    Numerical experiments validate the accuracy and efficiency of the method, including one-dimensional photonic quasicrystals and two- and three-dimensional QEOs.
	
	Future work will proceed in two directions. First, extending the framework to non-self-adjoint operators will broaden its applicability to a wider class of quasiperiodic systems. Second, because the fractal structure of continuous spectra in QEOs is closely related to exotic phenomena in photonic quasicrystals, we will couple the present numerical framework with physical models to guide the design of materials with tailored spectral properties.

\bibliographystyle{amsplain}
\bibliography{references}

\end{document}